\documentclass{amsart}
\usepackage{mathrsfs}
\usepackage[all]{xy}

\usepackage{wasysym}
\usepackage{amssymb}
\usepackage{comment}
\usepackage{mathtools}
\usepackage{mathbbol}

\usepackage{ifpdf}
\ifpdf 
  \usepackage[pdftex]{graphicx}
  \DeclareGraphicsExtensions{.pdf,.png,.jpg,.jpeg,.mps}
  \usepackage{pgf}
\else 
  \usepackage{graphicx}
  \DeclareGraphicsExtensions{.eps,.bmp}
  \usepackage{pgf}
  \usepackage{pstricks}
\fi
\usepackage{epic,bez123}
\usepackage{wrapfig}

\DeclareMathAlphabet{\mathpzc}{OT1}{pzc}{m}{it}

\newtheorem{mainthm}{}

\newtheorem{thm}{Theorem}[section]
\newtheorem{prop}[thm]{Proposition}
\newtheorem{cor}[thm]{Corollary}
\newtheorem{lem}[thm]{Lemma}

\newtheorem*{claim}{Claim}

\newtheorem{conv}[thm]{Convention} 

\theoremstyle{definition}
\newtheorem{defn}[thm]{Definition}

\theoremstyle{remark}

\newtheorem*{rem}{Remark}
\newtheorem{example}[thm]{Example}

\newtheorem*{ack}{Acknowledgment}

\DeclarePairedDelimiter\ceil{\lceil}{\rceil}

\newcommand{\e}[1]{\omega(#1)}

\newcommand{\diam }[1]{{\textbf{diam}\big(#1\big)}}
\newcommand{\proj }{\pi}

\newcommand{\len }{\ell}

\newcommand{\ax}{\mathbf{Ax}}
\newcommand{\Teichax}{\mathbf{ax}}
\newcommand{\tax}{\mathscr{A}}

\usepackage[bookmarks=true, pdfauthor={YANG Wenyuan}]{hyperref}

\begin{document}

\title{Counting conjugacy classes in groups with contracting elements}

\author{Ilya Gekhtman}
\address{Faculty of Mathematics, Technion-Israel Institute of Technology, Haifa, Israel  3200003}
\email{ilyagekh@gmail.com}
\author{Wen-yuan Yang}
\address{Beijing International Center for Mathematical Research (BICMR), Beijing University, No. 5 Yiheyuan Road, Haidian District, Beijing, China}
\email{yabziz@gmail.com}

\thanks{W. Y. is supported by the National Natural Science Foundation of China (No. 11771022). I.G. was partially supported by the National Science and Engineering Research Council of Canada.}



\date{\today}

\dedicatory{}


\begin{abstract}
In this paper, we derive an asymptotic formula for the number of  conjugacy classes of elements in a class of statistically convex-cocompact actions with contracting elements. Denote by  $\mathcal C(o, n)$ (resp. $\mathcal C'(o, n)$)   the set of (resp. primitive) conjugacy classes of algebraic length at most $n$ for a basepoint $o$. The main result is the following asymptotic formula:   
$$\sharp \mathcal C(o, n) \asymp \sharp \mathcal C'(o, n) \asymp \frac{\exp(\e Gn)}{n}.$$
 A similar formula holds for   conjugacy classes using stable length.  As a consequence of the formulae, the conjugacy growth series  is transcendental for all non-elementary relatively hyperbolic groups, graphical small cancellation groups with finite components. As by-product of the proof, we establish several useful properties for an exponentially generic set of elements.  In particular, it yields a positive answer to a question of J. Maher that an exponentially generic elements in mapping class groups have their Teichm\"{u}ller axis contained in the principal stratum.    
\end{abstract}

\maketitle

\setcounter{tocdepth}{1} \tableofcontents

\section{Introduction}
\subsection{Motivation}

Let $G$ be a countable group acting by isometries on a proper geodesic metric space $(\mathrm Y, d)$. Assume that the action of $G$ on $\mathrm Y$ is proper, so that for any basepoint $o \in \mathrm Y$, the set $N(o, n):=\{g \in G: d(o, go) \le n\}$  is  finite.  For any subset $X\subset G$, the \textit{growth rate} $\e X$ is defined as follows:
$$
\e X := \limsup\limits_{n \to \infty} \frac{\log \sharp X\cap N(o, n)}{n}.
$$
The quantity $\e X$ does not depend on the basepoint $o$.  We assume that $\e G$ also called critical exponent for the action is finite in this paper.

The motivating example is the   action of a finitely generated group $G$ on its Cayley graph $(\mathrm Y, d)$ with respect to  a finite generating set $S$. Here, $d$ is the word metric and, by a subadditive inequality, the critical exponent is a true limit, which is the \textit{growth rate} of $G$ with respect to $S$.

There has been considerable interest in studying the number of conjugacy classes of $G$ with stable translation length at most $n$ on $\mathrm Y$, in particular its asymptotics as $n\to \infty$ (see below for precise definitions). 

When $\mathrm Y$ is a negatively curved contractible manifold, the problem of counting conjugacy classes  is equivalent to counting closed geodesics in the quotient $\mathrm Y/G$. Indeed the conjugacy class of a loxodromic element $g\in G$  defines a closed geodesic on $\mathrm Y/G$ which is the image in $\mathrm Y/G$ of the translation axis of $g$, and  its stable length is precisely the length of the associated geodesic. The conjugacy class is called \textit{primitive} if $g$ is not a proper power of any element of $G$, in which case the associated closed geodesic is primitive. 
In his 1970 thesis \cite{MarThesis}, Margulis established a precise asymptotic formula for the set $\mathcal C'(n)$ of primitive closed geodesics with length less than $n$ as follows
$$
\sharp \mathcal C'(n) \sim \frac{\exp(\mathbf h n)}{\mathbf hn}
$$
where $\mathbf h=\e {\pi_1M}$ is the topological entropy of geodesic flow on the compact manifold $M$. Margulis' result has been generalized to various actions which display some hyperbolicity.
An analogous formula for closed geodesics has been obtained for the following actions:

\begin{enumerate}\label{mainexamples}

\item
Quotients of CAT($-1$) space by a geometrically finite groups of isometries admitting
a finite Bowen-Margulis-Sullivan measure on the unit tangent bundle,   \cite[Th\'eor\`eme 5.2]{Roblin}. 
\item
Compact rank-1 manifolds,   \cite{Knieper2}.

\item
The moduli space of closed orientable surfaces of genus $\ge 2$ endowed with the Teichm\"{u}ller metric (which is the quotient of Teichm\"{u}ller space by the action of the mapping class group),
 \cite{EM}, \cite{EMR}.

\item
Covers of the above moduli space associated to convex-cocompact subgroups of the above mapping class groups, \cite{Gekht}.
\end{enumerate}

Beyond the manifold setting, Coornaert and Knieper \cite{CoK1}, \cite{CoK2} proved that for hyperbolic groups acting on their Cayley graphs, the number of primitive conjugacy classes of stable length at most $n$ is up to a bounded multiplicative constant equal to $e^{hn}/n$ where $h$ is the exponential growth rate of the Cayley graph (see also Antolin and Ciobanu\cite{AntCio} for all conjugacy classes). 

Motivated by the work of Margulis,  Guba and Sapir \cite{GubaSapir}  initiated a systematic study of  conjugacy growth function in groups, namely  the number $\sharp \mathcal C(1, n)$ of conjugacy classes intersecting a ball centered at $1$ of radius $n$. Many examples of groups  were found to have exponential conjugacy growth $\sharp \mathcal C(1, n) \ge a^n$ for some $a> 1$, including non virtually solvable  linear groups \cite{BCLM},  non-elementary acylindrically hyperbolic groups \cite{HOsin}, and so on. We refer the reader to   \cite{GubaSapir} and many related references therein.

\subsection{Statement of main results}
The goal of the present paper is to establish coarse multiplicative asymptotic formulae describing growth of conjugacy classes  for a   more general class of actions called \textit{statistically convex-cocompact actions} with a \textit{contracting element}, encompassing the above examples.

The notion of a contracting element plays a significant role in counting conjugacy classes. An element $g\in G$ is called \textit{contracting} if for some (or any) basepoint $o\in \mathrm Y$,  the \textit{stable length} defined by  
$$
\tau(g):=\lim_{n\to\infty} \frac{d(o, g^no)}{n}
$$
is positive and the subset $\langle g\rangle \cdot o$ is contracting in $\mathrm Y$. Here a  subset $X\subset \mathrm Y$  is \textit{contracting} if it enjoys a certain property suggestive of negative curvature: namely  that any metric ball $B$ outside $X$ has the  uniformly bounded shortest projection to $X$. We remark that $\tau(g)>0$ if and only if the map $n\in \mathbb Z\mapsto g^no\in \mathrm Y$ is a quasi-isometric embedding.

To make precise the notion of counting conjugacy classes, we now describe two ways to assign a length to a conjugacy class: stable length and algebraic length.

By definition, $\tau(g)$ does not depend on the basepoint, so it is a conjugacy invariant. Denote by $[g]$ the set of elements conjugate to $g$. We then have a well-defined  function $\tau[g]$ on the set of conjugacy classes $[g]$. 

We  shall fix a basepoint $o\in \mathrm Y$ throughout.  The \textit{algebraic length} $\len_o[g]$ of a conjugacy class is defined as $\len_o [g]:=\inf\{d(o, g'o): g'\in [g] \}$. This clearly depends on the choice of the basepoint. By the subadditivity, we see that $$\tau[g]=\inf\{\frac{d(g^no, o)}{n}: n\ge 1\}.$$ 
Thus, for every $g\in G$, we   have $\tau[g]\le\len_o[g]$.

In many geometric examples, the stable length is realized by the algebraic length for a certain basepoint.  This is the case when  $\mathrm Y$ is the universal cover of a compact Riemannian manifold,   a CAT(0) space, and in Examples \ref{mainexamples}.

In \cite{YANG10}, the second-named author defined a class of statistically convex-cocompact actions, which is a dynamical generalization of convex-cocompact actions studied in many different settings.  Given constants $0\le M_1\le M_2$, let $\mathcal O_{M_1, M_2}$ be the set  of elements $g\in G$ such that there exists some geodesic $\gamma$ between $\gamma_-\in B(o, M_2)$ and $\gamma_+\in B(go, M_2)$ with the property that the interior of $\gamma$ lies outside $N_{M_1}(Go)$. 

\begin{defn}[statistically convex-cocompact action]\label{StatConvex}
If there exist two positive constants  $M_1, M_2>0$ such that $\e {\mathcal O_{M_1, M_2}} < \e G<\infty$, then the action of $G$ on $\mathrm Y$ is called \textit{statistically convex-cocompact (SCC)}.  

\end{defn} 

Among other results in \cite{YANG10}, let us point out that a SCC action with a contracting element has the \textit{purely exponential growth} (PEG) property
$$
\sharp N(o, n)\asymp \exp(\e G n),
$$
where $\asymp$ denotes the two sides differ by a multiplicative constant.

We say that an element $g\in G$ is \textit{primitive} if it cannot be written as a proper power $g=g_0^n$ for $|n|\ge 2$ and $g_0\in G$. If  $g$ is a contracting element, we  shall also  consider a stronger notion of primitivity.  Let $E(g)$ be the maximal elementary subgroup containing $g$. By a theorem of Stallings there exists a subgroup $E^+<E(g)$ of index at most two and a finite group $F$ with the following exact sequence
$$
1\to F\to E^+\stackrel{\phi}{\to} \mathbb Z{\to} 1.
$$
Any element $h\in E^+$ in the preimage $\phi^{-1}(\pm 1)$ is called \textit{strongly primitive}. A strongly primitive (contracting) element is primitive by Lemma \ref{StrongPrimitive}, but the converse may not be true.

Our main theorem  gives asymptotic formulae using algebraic length and stable length for primitive,  as well as for all, conjugacy classes. Let $$\mathcal C(o, n)=\{[g]: 0\le \len_o[g]\le n\},$$
$$\mathcal C(n)=\{[g]: 0<\tau[g]\le n\},$$  and $\mathcal {C'}(o, n) \subset \mathcal C(o, n)$ and $\mathcal {C'}(n)\subset \mathcal C(n)$ denote the subsets of primitive (or strongly primitive)  ones respectively.

\begin{mainthm}\label{mainthm}
Suppose that a non-elementary group $G$ admits a SCC action on a proper geodesic metric space $(\mathrm Y, d)$ with a contracting element and $\e G<\infty$.   Let $o\in\mathrm Y$ be a basepoint. There exists a constant $D=D(o)>0$ such that the following statements hold:
\begin{enumerate} 
 \item
$\displaystyle\sharp \mathcal C(o, n) \asymp_D \sharp \mathcal {C'}(o, n) \asymp_D \frac{\exp(\e G n)}{n}.$
\item
$\displaystyle{
\sharp \big(\mathcal C(n) \cap \mathcal C(o, n)\big) \asymp_D\sharp \big(\mathcal {C'}(n)\cap \mathcal C(o, n)\big) \asymp_D \frac{\exp(\e G n)}{n}}.
$

 \item
 $\displaystyle\frac{\sharp \big(\mathcal {C'}(n)\cap\mathcal C(o, n)\big)}{\sharp \big(\mathcal C(n)\cap \mathcal C(o, n)\big)} \to 1$ and $\displaystyle\frac{\sharp \mathcal {C'}(o, n)}{\sharp \mathcal C(o, n)} \to 1$ exponentially fast.
\end{enumerate} 
\end{mainthm}
From now on,  a formula as above in (1) (2) will be referred to as a \textit{prime conjugacy growth formula}.

\begin{rem}
Most of the forementioned results count conjugacy classes of bounded stable length. Our result also considers the algebraic length, which depends on the basepoint, in the setting of non-cocompact actions. We remark that  the conjugacy growth formula for algebraic length is not a direct consequence of the one using stable length.
 
\end{rem}

Regarding to the complicated formula in (2), we wish to explain some subtleties about counting   conjugacy classes with respect to stable length. 

For a general proper action,  measuring conjugacy classes in stable length has the weakness  that the conjugacy growth function might not be even  defined as there may be infinitely many conjugacy classes with bounded stable length. For example,  Conner \cite{Conner} constructed examples of groups with stable lengths accumulating at 0 (the opposite is called \textit{translation discrete} there). See \cite{IKap} for examples with discrete spectrum of stable lengths. 

Examples of relatively hyperbolic groups whose stable length spectrum is not discrete can be easily built  by taking a free product of two groups with this property. This forces us to consider the formula with stable length more carefully in \ref{mainthm}.  From a geometric point of view, one may wonder whether these conjugacy classes coming from parabolic subgroups are not interesting  since they degenerate on the quotient manifolds (e.g. geometrically finite hyperbolic manifolds). This perhaps motivates one to count only loxodromic conjugacy classes associated with closed geodesics. However, this  still does not correct the formula: there are examples of hyperbolic manifolds with fundamental groups satisfying the  SCC condition but containing infinitely many closed geodesics with bounded length.  See \textsection \ref{SCCinfty} for construction of such examples. Moreover, such examples could exist in the class of acylindrical actions on  hyperbolic spaces. See Lemma \ref{InftyConjClassesLem}.  

 
Nevertheless,  there are many new settings in which we can obtain a satisfactory formula with stable length.  
\subsection{Applications}

We first consider the class of CAT$(0)$ groups and its subclass of cubical groups. An isometry of a proper CAT$(0)$ space is contracting with respect to the CAT$(0)$ metric exactly when it has a rank-1 axis in the space, ie.: a geodesic which does not bound a flat half-plane (\cite[Theorem 5.4]{BF2}). If the space is a CAT$(0)$ cube complex, such an isometry is also contracting with respect to the $\ell^1$-metric on the $1$-skeleton, which coincides with the standard word metric for a right-angled Artin or Coxeter group. 
\begin{cor}[CAT(0) groups]\label{CAT0Thm} 

\begin{enumerate}
\item
A non-elementary group $G$ acting geometrically on a CAT$(0)$ space $X$ with a rank-1 element satisfies the prime conjugacy growth formulae for algebraic length and stable length with respect to the $CAT(0)$ metric on $X$.

\item
A non-elementary group $G$ acting geometrically on a CAT$(0)$ cube complex  with a rank-1 element satisfies the prime conjugacy growth formulae for algebraic length and stable length with respect to the $\ell^1$ metric on the $1$-skeleton. 

In particular, this holds for Right angled Artin or Coxeter group with the standard word metric, provided that the group is not virtually a product of nontrivial groups. 
 
\end{enumerate}
\end{cor}
\begin{rem}
As mentioned in Examples \ref{mainexamples}, the analogue of   (1) for smooth manifolds of non-positive curvature is due to Knieper \cite{Knieper2}. His method uses conformal densities on the boundary, whereas our methods are completely geometric and elementary and do not involve any measure theory.

The formula for algebraic length is immediate by  \ref{mainthm}. But for the stable length, it needs an additional ingredient that for \textit{every} rank-1 element, its stable length coincides with algebraic length up to a uniform error.  This might not be true in other classes of groups. 
\end{rem}

Given a relatively hyperbolic group $(G, \mathcal P)$,  a \textit{hyperbolic} element is by definition an infinite order element not conjugated into any subgroup in $\mathcal P$. For a hyperbolic element, we note that the stable length coincides with algebraic length, up to a uniform error. We thus obtain the following corollary.

\begin{cor}\label{RHThm}
Let $(G, \mathcal P)$ be a relatively hyperbolic group. Then for the action on the Cayley graph, the prime conjugacy growth formulae holds for all conjugacy classes with algebraic length, and for conjugacy classes of contracting elements with stable length.
\end{cor}

We note that the class of  the cubical groups and relatively hyperbolic groups contains many groups which are not Gromov hyperbolic. 
We next consider the class of graphical small cancellation groups; this is a class containing many groups which are not relatively hyperbolic, see \cite{GruS} for many examples.
Graphical small cancellation groups contain contracting elements with respect to actions on their Cayley graphs by \cite[Theorem 5.1]{ACGH} and thus \ref{mainthm} applies to count their conjugacy classes.
\begin{cor}\label{GSCThm}
The prime conjugacy growth formula for  algebraic length holds for graphical small cancellation groups with finite components on the Cayley graph with respect to small cancellation presentation.
\end{cor}
\begin{rem}
By   \cite[Lemma 5.3]{ACGH}, some power of every contracting element preserves a geodesic. However, we do not know whether the stable length can be realized by algebraic length, without raising the element to a power.  So a formula for stable length is not available. 
\end{rem}


Since the stable length of a pseudo-Anosov element coincides with the length of a closed geodesic, \ref{mainthm} applies to count closed geodesics on certain covers of moduli space corresponding to subgroups acting by a SCC action on Teichm\"{u}ller spaces. Convex-cocompact subgroups in the sense of Farb and Mosher \cite{FarbMosher} are obviously SCC, for which the first-named author \cite{Gekht} previously obtained a  precise formula  of closed geodesics. In \cite[Proposition 6.6]{YANG10}, examples of non-convex-cocompact SCC actions are constructed out of  subgroups generated by disjoint Dehn twists and a sufficiently high power of a pseudo-Anosov element.

\begin{cor}\label{ModThm}
The prime conjugacy growth formula holds  for closed geodesics on the cover of the moduli space associated with subgroups in mapping class groups on Teichm\"{u}ller space constructed in \cite{YANG10}.
\end{cor}

We remark that this corollary is  not an immediate consequence of \ref{mainthm}. It requires an additional fact about the subgroups $\Gamma<Mod(S)$ at hand which is evident from the construction in \cite{YANG10}, namely that the Teichm\"{u}ller geodesic axis of any pseudo-Anosov $g\in \Gamma$ is contained within a bounded distance of an orbit of $\Gamma$. See Section 8 for more details.

 We can conclude Corollary \ref{ModThm} from the following general statement. 
 Compare with \cite[Th\'eor\`eme 5.1.1]{Roblin} and  examples in \textsection \ref{SCCinfty}.

\begin{thm}\label{CompactSupport}
Under the assumption of \ref{mainthm}, assume in addition that every contracting element preserves a geodesic axis. Then the   prime conjugacy growth formula with stable length holds for all contracting elements with axis intersecting a fixed finite neighborhood of the  orbit $Go$.  
\end{thm}
\begin{rem}[moduli space]
On one hand,   Hamenstädt \cite{Hamen2} and Rafi \cite{Rafi}  proved     that there are closed geodesics outside every compact part of moduli space; on the other hand, Eskin and Mirzakhani \cite{EM} showed that     the number of closed geodesics outside a certain compact part is exponentially small relative to the ones intersecting it.

If an analogue of the latter holds for the SCC covers of moduli space, then \ref{mainthm} allows to count {all} closed geodesics on \textit{any} SCC cover as in Corollary \ref{ModThm}.
\end{rem} 
  
\textbf{Applications to   conjugacy growth series}.
There is some recent interest in understanding the complexity of the following formal \textit{conjugacy growth series} for a   basepoint $o\in G$: 
$$
\mathcal P(z)=\sum_{[g]\in  G}     z^{\len_o[g]}  \in \mathbb Z[[z]],
$$ in particular
whether it is rational, algebraic, or transcendental over $\mathbb Q(z)$. One could similarly look at formal series obtained from counting primitive conjugacy classes and using stable length. The same  result stated below holds for them as well.

This series is naturally stated with respect to the action on the Cayley graph, where $o$ is the identity  and $\len_1[g]$ is the minimal length of elements in $[g]$. It is well-known that,  if counting $N(1, n)$ instead of $\mathcal C(1, n)$, the formal growth series $\sum_{g\in  G}  z^{d(1, g)}$ is rational for any hyperbolic group (cf. \cite{Cannon}). However, in \cite{Rivin2}\cite{Rivin3}, Rivin computed the formal conjugacy growth series for free groups with word metric which  turns out to be irrational.

Furthermore, Rivin \cite{Rivin2} conjectured that the conjugacy growth series of a hyperbolic group is rational if and only if it is virtually cyclic. In  \cite{CHFR}, Ciobanu, Hermiller, Holt and Rees proved that a virtually cyclic group has rational  conjugacy growth series. Later on, Antolin and Ciobanu \cite{AntCio} established the other direction of Rivin's conjecture by showing that a non-elementary hyperbolic group has transcendental conjugacy growth series. The main ingredient is a prime conjugacy growth formula for all conjugacy classes in hyperbolic groups, which extends earlier work of Coornaert and Knieper \cite{CoK1}, \cite{CoK2}. Hence, by the same reasoning, we obtain the following consequence.

\begin{thm}\label{CongugacySeries}
Let $G$ be a non-elementary group acting on a proper geodesic space $(\mathrm Y,d)$ with a contracting element. Then the conjugacy growth series is transcendental. 
\end{thm}

In \cite{Sisto}, Sisto proved that if a group admits a proper action with a contracting element then it must be acylindrically hyperbolic in the sense of Osin \cite{Osin6}. In particular, we have the following corollary.
\begin{cor}
Let $G$ be a non-elementary group with a finite generating set $S$. If $G$ has a contracting element with respect to the action on the corresponding Cayley graph, then the conjugacy growth series is transcendental.
\end{cor}
This   confirms Rivin's conjecture for a large subclass of  acylindrically hyperbolic groups, including right-angled Artin/Coxeter groups, relatively hyperbolic groups and graphical small cancellation groups, etc.

In the class of relatively hyperbolic groups,  the conclusion actually holds for every generating set. This is a direct generalization of the corresponding statement in \cite{AntCio} for hyperbolic groups.

\begin{cor}
The conjugacy growth series of a non-elementary relatively hyperbolic group with respect to any finite generating set  is  transcendental.
\end{cor}
 
To conclude the introduction of our results, we mention a result of independent interest which is  a by-product of the proof of our main theorem. 

 {Let $E(g)$ be the maximal elementary subgroup of $G$ containing $g$. Define the set $\tax(g)=E(g)\cdot o$  to be the coarse axis of a contracting element $g$. We define a $\langle g\rangle$ invariant subset $\tax_R(g)$ of $\mathrm Y$ to be an \textit{$R$-stable axis} of the element $g$, if  for  any point $x\in \tax_R(g)$, the ball $B(x, R)$ intersects any bi-infinite  geodesic $\alpha$ which is contained a finite neighborhood of $\tax(g)$ and $d(x, gx)>3R$}. A simplified version of Theorem \ref{LinearDrift} is stated below.
\begin{thm}\label{simplifiedLinearDrift}
Assume that $G$ admits a SCC action on a proper geodesic metric space $(\mathrm Y, d)$. Fix a contracting element $f$. Then there exist $R>0$ depending on $f$ such that for  any  $1>\theta_1, \theta_2>0$ and any integer $m>1$,  the set of elements $g$ with $n=d(o, go)$ satisfying
\begin{enumerate}
\item
$n\ge \tau [g]  \ge (1-\theta_1) n$,
\item
$ d(o, \tax_R(g)) \le n \theta_2$,
\item
any  bi-infinite geodesic which is contained in a finite neighborhood of $\ax(g)$ contains an $(\epsilon ,f^m)$-barrier (see Def. \ref{barriers}),
\end{enumerate}
is exponentially generic.  
\end{thm}

Examining the axis of a pseudo-Anosov element in Teichm\"{u}ller space, we derive an application to mapping class groups which gives a positive answer to the (first part of) the question posed by J. Maher in \cite[Question 6.4]{DHM}. By abuse of language, we denote below by $\Teichax(g)$  the Teichm\"{u}ller axis of a pseudo-Anosov element of $g$ both as a subset of Teichm\"{u}ller space and its unit tangent bundle.  

Recall that a pseudo-Anosov element is contracting with respect to  Teichm\"{u}ller metric, cf. \cite{Minsky}. Moreover, the action of $MCG(S)$ on $Teich(S)$ is  statistically convex-cocompact (see \cite[Theorem 1.7]{EMR},  \cite[Section 10]{ACTao}).

{Let $\mathbf{PS}$ be the principal stratum of quadratic differentials}.
\begin{thm}\label{Teichprincipal}
Let the mapping class group $G$ act  on the Teichm\"{u}ller space $(\mathrm Y, d)$ endowed with Teichm\"{u}ller metric. Then  for  any  $0<\theta_1, \theta_2<1$, we have 
 $$
 \{g\in G:    n\ge \tau [g] \ge \theta_1 n \;\&\; d(o, \Teichax(g)) \le \theta_2 n\;\&\; \Teichax(g) \in \mathbf{PS}~\text{where}~n=d(o,go) \}
$$
is exponentially generic.
\end{thm}

\subsection{Outline of the proof of Main Theorem}\label{OutlinePF}

Recall that  $N(o, n)=\{g: d(o, go)\le n\}$ is the set of elements in a ball of radius $n$. When the action is SCC, we have 
$$
\sharp N(o, n) \asymp \exp(\e Gn).
$$
Since   $\len_o[g]\ge \tau[g]$ for any   $o\in \mathrm Y$, the following relations are basic in our discussion:
$$
\mathcal C (o, n) \subset \mathcal C(n),\;\; \mathcal C (o, n) \subset N(o, n).
$$

The key idea of the proof is to choose an exponentially generic  set of contracting elements with locally uniform hyperbolic properties. Such   properties are encapsulated in the following notion (See Def. \ref{barriers} for a precise definition).   

\begin{defn}
With a basepoint $o\in \mathrm Y$ fixed, an element $h\in G$ is called \textit{$(\epsilon, M, g)$-barrier-free} if there exists an \textit{$(\epsilon, f)$-barrier-free} geodesic $\gamma$ with $\gamma_-\in B(o, M)$ and $\gamma_+\in B(ho, M)$:  there exists no $t\in G$ such that $d(t\cdot o, \gamma), d(t\cdot fo, \gamma)\le \epsilon$. 
\end{defn}

In \cite{YANG10}, the second-named author proved that for any $f\in G$, the set of $(\epsilon, M, f)$-barrier-free elements is exponentially negligible for a constant $M>0$ appearing in Definition \ref{StatConvex}. As fore-mentioned, a contracting element represents a sense of hyperbolicity in direction.  Therefore, once a contracting element $f$ is provided and fixed through out, the geodesic with $(\epsilon, M, f)$-barriers \textbf{uniformly} behaves like a geodesic in a Gromov-hyperbolic space when coming to the barriers. This often allows us to run arguments via hyperbolic geometry (e.g. see Proposition \ref{StablePointedLength}). 

Furthermore, the set of minimal representatives in conjugacy classes of non-contracting elements are barrier-free, so exponentially negligible as well. We show that  $\mathcal C(o, n)$ has the   growth rate $\e G$;  it is sufficient to count conjugacy classes of contracting elements.

The next  step is to compute the algebraic length and stable length of a contracting element. The algebraic length is a bit easier to estimate from the definition. However, giving a uniform way to estimate the stable length of every contracting element seems to be hard, if not impossible. The solution is here that we can estimate the  stable length for an exponentially generic set of contracting elements.


\begin{lem}[Corollary \ref{GenericSetTwoLengths}, Stable length $\simeq$ algebraic length]
There exist an exponentially generic set $\mathcal G$ of contracting elements  and a   constant $D=D(f_1, f_2, o)>0$ for two independent contracting elements $f_1,f_2$ such that for each $g\in \mathcal G$, the following holds
$$
0\le \len_o[g]-\tau[g]\le D.
$$
\end{lem}

Therefore we do not need to distinguish between stable length and algebraic length, and so $\sharp \mathcal C(o, n)$ and $\sharp \large (\mathcal C(n)\cap \mathcal C(o, n)\large)$ are coarsely equal.

\paragraph{\textbf{1. Upper bound on strongly primitive conjugacy classes.}} To get the upper bound of $\mathcal C(o, n),$ we follow a piece of argument in \cite{CoK2} which works for a cocompact action of a (hyperbolic) group. Namely, take a conjugacy class $[g]\in \mathcal C(o, n)$ and if $g$ is strongly primitive and write $g=s_1\cdot s_2\cdots s_n$, then by cyclic permutation we obtain $n$ words in the same conjugacy class. If all cyclic permutations represent  distinct elements, then the upper bound of  $\mathcal C(o, n)$ is obtained as follows:
$$
\sharp \mathcal C'(o, n) \cdot n \le \sharp N(o, n).
$$
However, this argument breaks {down} when the action is not cocompact. The reason is that since $[o, go]$ may have large proportion outside the orbit $Go$,  there is no way to write $g$ as a product of a  {number of elements linear in $d(o,go)$}. We overcome this   by showing that geodesic segments associated to generic  contracting elements spend a definite proportion of the time in $N_M(Go)$.

\begin{lem}[Corollary \ref{LargePercentGeneric}, Thick contracting elements]
There exists an exponentially generic set $\mathcal G$ of contracting elements such that for each $g\in \mathcal G$, we have
$$
\len([o, go]\cap N_M(Go)) \ge 0.9 \cdot \len([o, go]).
$$
\end{lem} 

Thus, the upper bound on the number of strongly primitive conjugacy classes is obtained in Corollary \ref{PrimConjUpBnd}.

\paragraph{\textbf{2. Strongly non-primitive ones are growth tight.}} {Counting strongly non-primitive conjugacy classes will require more effort}. {The idea however is simple: the number of non-primitive conjugacy classes is exponentially negligible, compared to the primitive ones.} The proof of this result, Lemma \ref{GenericPrim}, is intuitively clear, since a non-primitive element $g$ is a proper power of a primitive element $g_0$. An inspection of the argument shows a difficulty as follows. 

Let $h=h_0^m f$ be a non-strongly primitive element in $E(g)$ for $|m|\ge 2$ and $f\in F$. Define a map $\Pi$ sending $[h]$ to $[h_0]$. Clearly, the image of $\Pi$ is   exponentially negligible. But there is no reason that $\Pi$ is uniformly finite to one, since the size of $F$ can change. The following result fixes this issue.
 
\begin{lem} [Lemma \ref{UnifKernel}, Uniform kernel]
There exist  an exponentially generic set $\mathcal G$ of contracting elements and an integer $N>0$ such that for each $g\in \mathcal G$, we have
$$
1\to F \to E^+(g)\stackrel{\phi}{\to} \mathbb Z{\to} 1
$$
and $\sharp F\le N$.
\end{lem}

The proof of the lemma relies on the very recent work of Bestvina, Bromberg, Fujiwara and Sisto \cite{BBFS} improving the {earlier} work \cite{BBF} so that the action on the projection complex is acylindrical hyperbolic. We then prove that an exponentially generic set of elements act by loxodromic isometries on the projection complex. Then a result of Osin \cite[Lemma 6.8]{Osin6} concludes the proof.
 
From this result, we show that non-primitive conjugacy classes are exponentially negligible (Lemma \ref{GenericPrim}). Hence,  the upper bound for all conjugacy classes is proved in Corollary \ref{ConjUpBnd}.  

\paragraph{\textbf{3. Lower bound on conjugacy classes.}} The lower bound is by construction. By \cite{YANG10} (recalled in Lemma \ref{ClosingLemma}), there exists a maximal separated set $T$ in $A(o, n, \Delta)$ and a contracting element $f$ such that $T\cdot f$ consists of contracting elements {and has} the same cardinality as $T$. Following an argument of Coornaert and Knieper \cite{CoK1}, we show that each conjugacy class $[g]$  contains at most $\theta n$ elements in $T\cdot f$ for some uniform number $\theta>0$. Thus, we constructed at least $\displaystyle\frac{\exp(\e G n)}{n}$ conjugacy classes (see Corollary \ref{LbdConjGrowth}). Finally, the lower bound for primitive conjugacy classes  is a direct consequence of the growth tightness  of non-primitive ones mentioned above (Corollary \ref{LbdPrimConjGrowth}).

\paragraph{\textbf{The structure of the paper is as follows.}} {The preliminary   \textsection \ref{Section2} introduces the contracting property, and a class of periodic admissible paths. The core of   \textsection \ref{Section3} is Proposition \ref{StablePointedLength}  identifying the stable length with algebraic length.  Along the way,   the linear growth is proved in Theorem \ref{simplifiedLinearDrift}. Many conjugacy classes are constructed in \textsection \ref{Section4}, establishing the lower bound. Then \textsection \ref{Section5} deals with non-cocompact actions and obtains the upper bound for primitive conjugacy classes. Section \textsection \ref{Section6} addresses the issue of unbounded torsion. With previous ingredients in hand, the proof of the main theorem is completed in \textsection \ref{Section7}. The final   \textsection \ref{SecProofs} explains the applications to several specific classes of groups.}        

\ack
The authors are grateful to the referee for many  remarks, corrections and suggestions which  greatly improved the exposition of the paper. W. Y. is supported by the National Natural Science Foundation of China (No. 12131009).

\section{Preliminaries}\label{Section2}

\subsection{Notations and conventions}\label{ConvSection}
Let $(\mathrm Y, d)$ be a proper geodesic metric space.  Given a point $y \in Y$ and a closed subset $X \subset \mathrm Y$,
let $\pi_X(y)$ be the set of points $x$ in $X$ such that $d(y, x)=d(y, X)$. The \textit{projection} of a subset
$A \subset \mathrm Y$ to $X$ is then $\pi_X(A): = \cup_{a \in A} \pi_X(a)$. Whenever talking about projection, we shall assume the closedness of the subset $X$ under consideration so that $\pi_X(A)$ is nonempty.

Denote $d_X(Z_1, Z_2):=\diam{\pi_X({Z_1\cup Z_2})}$, which is the diameter of the projection of the union $Z_1\cup Z_2$ to $X$. So $d_X^\pi(\cdot, \cdot)$ satisfies the triangle inequality
$$
d_X^\pi(A, C) \le d_X^\pi(A, B) +d_X^\pi(B, C).
$$ 
{We use $d_X(Z):=\diam{\pi_X(Z)}$ as well in the sequel.}

We always consider a rectifiable path $\alpha$ in $\mathrm Y$ with arc-length parametrization.  Denote by $\len (\alpha)$ the length
of $\alpha$, and by $\alpha_-$, $\alpha_+$ the initial and terminal points of $\alpha$ respectively.   Let $x, y \in \alpha$ be two points which are given by parametrization. Then $[x,y]_\alpha$ denotes the parametrized
subpath of $\alpha$ going from $x$ to $y$. We also denote by $[x, y]$ a choice of a geodesic between $x, y\in\mathrm Y$.  
\\
\paragraph{\textbf{Entry and exit points}} Given a property (P), a point $z$ on $\alpha$ is called
the \textit{entry point} satisfying (P) if $\len([\alpha_-, z]_\alpha)$ is
minimal   among the points
$z$ on $\alpha$ with the property (P). The \textit{exit point} satisfying (P) is defined similarly so that $\len([w,\alpha_+]_\alpha)$ is minimal.

A path $\alpha$ is called a \textit{$c$-quasi-geodesic} for $c\ge 1$ if the following holds 
$$\len(\beta)\le c \cdot d(\beta_-, \beta_+)+c$$
for any rectifiable subpath $\beta$ of $\alpha$.

Let $\alpha, \beta$ be two paths in $\mathrm Y$. Denote by $\alpha\cdot \beta$ (or simply $\alpha\beta$) the concatenated path provided that $\alpha_+ =
\beta_-$.

Let $f, g$ be real-valued functions with domain understood in
the context. Then $f \prec_{c_i} g$ means that
there is a constant $C >0$ depending on parameters $c_i$ such that
$f < Cg$.  The symbols    $\succ_{c_i}  $ and $\asymp_{c_i}$ are defined analogously. For simplicity, we shall
omit $c_i$ if they are   universal constants. We also denote $f\simeq_{c_i} g$ if $|f-g|\le C$
\subsection{Contracting property}
\begin{defn}[Contracting subset]\label{ContrDefn}
Let $\mathcal {QG}$ denote a preferred collection of quasi-geodesics in
$\mathrm Y$. For given $C\ge 1$, a subset $X$ in $\mathrm Y$ is called $C$-\textit{contracting} with respect to $\mathcal {QG}$ if for any quasi-geodesic $\gamma \in \mathcal {QG}$ with $d(\gamma,
X) > C$, we have
$$d_{X} (\gamma)  \le C.$$
 A
collection of $C$-contracting subsets is referred to
as a $C$-\textit{contracting system} (w.r.t.
$\mathcal {QG}$).
\end{defn}
 
\begin{example} \label{examples} We note the following examples in various contexts.
\begin{enumerate}
\item
Quasi-geodesics and quasi-convex subsets are contracting with respect
to the set of all quasi-geodesics in hyperbolic spaces. 
\item
Fully quasi-convex subgroups (and in particular, maximal parabolic
subgroups) are contracting with respect to the set of all
quasi-geodesics in relatively hyperbolic groups (see Proposition
8.2.4 in \cite{GePo4}).
\item
The subgroup generated by a hyperbolic element is contracting  with
respect to the set of all quasi-geodesics in groups with non-trivial
Floyd boundary.  This is described in \cite[Section 7]{YANG6}.
\item
Contracting segments in CAT(0)-spaces in the sense of in
 Bestvina and  Fujiwara are contracting here with respect to the set of
geodesics (see Corollary 3.4 in \cite{BF2}).
\item
The axis of any pseudo-Anosov element is contracting relative to geodesics by   Minsky \cite{Minsky}.
\item
Any finite neighborhood of a contracting subset is still contracting
with respect to the same $\mathcal {QG}$.
\end{enumerate}
\end{example}

\begin{conv}\label{ContractingConvention}
In view of Examples \ref{examples}, the preferred collection
$\mathcal {QG}$ in the sequel will always be the set of all
geodesics in $\mathrm Y$.   
\end{conv}


We collect a few properties that will be used often later on. The proof   is a straightforward application of the contracting property, and is left to the interested reader.
\begin{prop}\label{Contractions}
Let $X$ be a contracting set.
\begin{enumerate}
\item
\label{qconvexity}  $X$ is \textit{$\sigma$-quasi-convex} for a function $\sigma: \mathbb R_+ \to \mathbb R_+$: given $c \ge
1$,  any $c$-quasi-geodesic with endpoints in $X$ lies in the
neighborhood $N_{\sigma(c)}(X)$. 
\item
 \label{nbhd}  Let $Z$ be a set with finite Hausdorff distance to $X$. Then $Z$ is contracting.

\end{enumerate}
\end{prop}






In most cases, we are interested in a  contracting system with the \textit{$\mathcal R$-bounded intersection} property for a function $\mathcal R: \mathbb R_{\ge 0}\to \mathbb R_{\ge 0}$ so that
$$\forall X\ne X'\in \mathbb X: \;\diam{N_r (X) \cap N_r (X')} \le \mathcal R(r)$$
for any $r \geq 0$. This property is, in fact, equivalent to the \textit{bounded intersection
property} of $\mathbb X$:  there exists a constant $B>0$ such that the
following holds
$$d_{X'}(X) \le B$$
for $X\ne X' \in \mathbb X$. See \cite[Lemma 2.3]{YANG6} for a proof of equivalence.
 
 Recall that $G$ acts properly on a proper geodesic metric space $(\mathrm Y, d)$.   An  element $h \in G$ is called  
\textit{contracting} if the orbit $\langle h \rangle\cdot o$ is contracting, and the orbital map
\begin{equation}\label{QIEmbed}
n\in \mathbb Z\to h^no \in \mathrm Y
\end{equation}
is a quasi-isometric embedding. Thus, any root or power of a contracting  element is contracting. Note that the set of contracting elements is preserved under conjugacy.

Given a contracting element $h$, we define a group 
\begin{equation}\label{Ehdefn}
E(h):=\{g\in G: d_H(\langle h\rangle o, g\langle h\rangle o)<\infty\}
\end{equation} 
where $d_H$ denotes the Hausdorff distance. By \cite[Lemma 2.11]{YANG10}, $E(h)$ is the unique maximal elementary group  containing $\langle h \rangle$ as a finite-index subgroup. Moreover, it can be described as follows, 
$$
E(h)=\{g\in G: \exists n > 0,\; (gh^ng^{-1}=h^n)\; \lor\;  (gh^ng^{-1}=h^{-n})\}.
$$

In what follows, the contracting subset 
\begin{equation}\label{axisdefn}
\ax(h)=\{f \cdot o: f\in E(h)\}
\end{equation} will be called the \textit{coarse axis} of $h$.  Hence, the collection $\{g \ax(h): g\in G\}$ is a contracting system with bounded intersection (by \cite[Lemma 2.11]{YANG10}).  Compare with combinatorial axis in Definition \ref{PeriAdmDef} and stable axis in Definition \ref{StabAxisDefn}.

 Two contracting elements $h, k\in G$ are \textit{independent} if the collection of contracting sets $\{g\ax(h),\; g\ax(k): g\in G\}$ has bounded intersection. Equivalently, they are independent if $E(h)$ and $E(k)$ are not conjugate  in $G$.
 
For $i=1, 2$, two  (oriented)  geodesics $\gamma_i: \mathbb R\to \mathrm Y$  have the \textit{same orientation}  if  $$
d_{H}(\gamma_1([o, +\infty]), \gamma_2([o, +\infty])) <\infty.
$$
An element $g\in G$ \textit{preserves the orientation} of an oriented geodesic $\alpha$ if $\alpha$ and $g\alpha$ have the same orientation.

The following fact is elementary.
\begin{lem}\label{GeodesicNearAxis}
For every contracting element $g\in G$,  there exists a bi-infinite geodesic $\alpha$ in a finite neighborhood of the axis $\ax(g).$ Moreover, the element $g$ preserves the orientation of the geodesic $\alpha$.
\end{lem}
\begin{proof}
By Proposition \ref{Contractions}, the first statement  follows from the quasiconvexity of the contracting subset $\ax(g)$ by a Cantor diagonal argument. To prove the moreover statement, fix an orientation of $\alpha: \mathbb R\to   \mathrm Y$, and a basepoint $o=\alpha(0)\in \alpha$. Let $R>0$ be the Hausdorff distance between $\alpha$ and $\ax(f)$. Thus, for $x_n=\alpha(n)$, there exists $k_n\in \mathbb Z$ such that $d(g^{k_n} o, x_n)\le R$. 

Suppose to the contrary that $R'=d_{H}(\alpha([0, +\infty]), g\alpha([0, -\infty])) <\infty$. Thus, for every $n>0$, there exists $m<0$ such that $d(g x_n, x_m)\le R'$. Then $d(g^{k_n+1}o, g^{k_m}o)\le 2R+R'$. Since the action is proper, we obtain that $k_n-k_m=n_0$ for infinitely many $n>0, m<0$. However, this is a contradiction  since  $d(g^{k_n}o, g^{k_m}o )\ge d(x_n, x_m)-2R$  tends $\infty$ when $|n|, |m|\to+\infty$.  
\end{proof}

\begin{cor}\label{EPlush}
For a contracting element $h$, let $E^+(h)$ be the subgroup of orientation preserving elements in $E(h)$. Then $E^+(h)=\{g\in G: \exists n > 0,\; gh^ng^{-1}=h^n\}$ is of index at most two and contains all contracting elements in $E(h)$.
\end{cor} 
\begin{rem}
A similar statement appears in Corollary 6.6 in \cite{DGO}, where the space $\mathrm Y$ is assumed to be $\delta$-hyperbolic.
\end{rem}

Since $E^+(h)$ is two-ended,   a theorem of Stallings implies the following exact sequence
\begin{equation}\label{ExactEg}
1\to F\to E^+(h) \stackrel{\phi}{\to} \mathbb Z\to 1,
\end{equation}
{where $F$ is a finite group} and $\mathbb Z$ is the group of integers.
Any element $g\in E^+(h)$ in $\phi^{-1}(\pm1)$ is called \textit{strongly primitive}.

\begin{lem} \label{StrongPrimitive}
A strongly primitive contracting element is a primitive element. 
\end{lem}

\begin{proof}
Let $g \in E(h)$ be  a strongly primitive contracting element. If $g$ is not primitive, then $g=g_0^k$ for some $|k|\ge 2, g_0\in G$. Since $g$ is contracting, it is readily checked that $\langle g_0 \rangle \cdot o$ is a contracting quasi-geodesic, which  implies  that $g_0$ is a contracting element. By definition of $E(h)$, we have $g_0\in E(h)$ and then $g_0\in E^+(h)$ by Lemma \ref{GeodesicNearAxis}. We obtain now  $\phi(g)=k \cdot \phi(g_0)\ne\pm 1$, contradicting the strong primitivity of $g$. 
\end{proof}
\begin{rem}
Note that the converse may not be true. However, a primitive but non-strongly primitive contracting element $g\in E(h)$ can be written as $g_0^k f$ for $|k|\ge 2$ and $f\in F\setminus \{1\}$, where $g_0$ is strongly primitive.
\end{rem}

\subsection{Periodic admissible paths}
In \cite{YANG6}, the notion of an admissible path is defined with respect  to   a  contracting system $\mathbb X$ in $\mathrm Y$. In the present paper, we shall focus on a particular class of admissible paths, which will serve as  coarse axes for contracting elements.

\begin{defn}[Periodic Admissible Path]\label{PeriAdmDef} 
Let $D, \tau>0$ and $\mathbb X$ be a contracting system with bounded projection.   Given  an element $g\in G$, a bi-infinite  path $\gamma =\cup_{i\in \mathbb Z} (q_i p_i)$ is called \textit{periodic $(g, D, \tau)$-admissible path} if the following hold
\begin{enumerate}
\item
each $p_i$ is a geodesic of length at least $D$ with two endpoints in $X_i \in \mathbb X$,  
\item
each  $q_i$ is a geodesic with $\tau$-bounded projection to $X_i$ and $X_{i-1}$:
$$
\max\{d_{X_i}(q_i), d_{X_{i+1}}(q_i)\}\le \tau,
$$
 
\item
  $q_ip_i=g^i(q_0p_0)$ for $i\in \mathbb Z$. 
\end{enumerate}
The collection  of $X_i\in \mathbb X$ is called a \textit{combinatorial axis} of the element $g$, denoted by $\mathbb A(g)$.
\end{defn}
 
\begin{rem}
The main difference with the definition in \cite{YANG6} is the third item which gives the way to represent   the admissible path periodically. For large $D\gg 0$, it is easy to see that we must have $X_i=g^iX_0$ by bounded intersection of $X_i\in \mathbb X$.

 The collection $\mathbb A(g)$ was called the \textit{saturation} of $\gamma$  in \cite{YANG10}.   Here it can be thought of as an axis  with respect to the action of $G$ on the projection complex built from the collection $\mathbb X$, where  each $X\in \mathbb X$ is collapsed to one vertex. Compare with a similar notion in \cite[Prop 3.26]{BBF} and see Lemma \ref{LoxdroPC}.
\end{rem}


Since a periodic   admissible path is a special case of the more general notion of an admissible path, the results proved in \cite{YANG6} and \cite{YANG10} apply here. We summarize the properties of periodic admissible paths as follows. 
 
\begin{prop}\label{admissible} 
Let $\mathbb X$ be a $C$-contracting system with $\mathcal R$-bounded projection. For any $\tau>0$, there are constants $D=D(\tau, {C},  {\mathcal R}), \epsilon = \epsilon(\tau, {C},   {\mathcal R})>0$ such that the following
holds.

Let    $g\in G$ be an element admitting a periodic $(g, D, \tau)$-admissible path $\gamma=\cup_{i\in \mathbb Z} (q_i p_i)$. Then
\begin{enumerate}
    \item 
    $g$ is a contracting element. 
    \item
    for any $i<j$ we have $\pi_{X_i}(X_j)\subset B((p_i)_+,\epsilon)$ and  $\pi_{X_j}(X_i)\subset B((p_j)_-,\epsilon)$.
    \item
    for any bi-infinite geodesic   $\alpha$ in a finite neighborhood   of  $\gamma$, we have that $\alpha$ intersects the $\epsilon$-neighborhood of $(p_{i})_-$ and $(p_{i})_+$.
\end{enumerate}

 
\end{prop}
\begin{proof}
By Proposition 2.9.2 in \cite{YANG10} we know that $\gamma$ is a contracting quasi-geodesic, so $g$ is contracting by definition. 

The assertion (3) was proved in \cite{YANG10} where $\gamma$ is a finite path with the same endpoints as $\alpha$. This  followed from  \cite[Corollary 3.7]{YANG6} of the bounded projection to $X=X_i$ from both sides: there exists $B=B(\sigma,  {\mathcal R})>0$ such that 
\begin{equation}\label{2sidesbddprojEQ}
\max\{\proj_{X}(\beta_1), \proj_{X}(\beta_2)\}  \le B
\end{equation}
where $\beta_1$ is the left one-sided infinite subpath of $\gamma$ issuing from $(p_i)_-$ and $\beta_2$ is the right one from $(p_i)_+$. 

\begin{figure}[htb] 
\centering \scalebox{0.7}{
\includegraphics{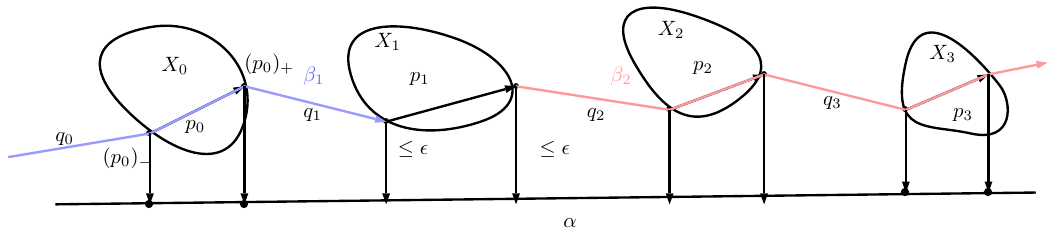} 
} \caption{Fellow travel periodic admissible paths.} \label{figure0}
\end{figure}

In the current setting, we first note  $\alpha \cap N_C(X)\ne\emptyset$. Indeed, if $\alpha$ is disjoint from $N_C(X)$, then $d_{X}(\alpha)\le C$ by contracting property of $X$. Say for some $R>0$, $\gamma$ stays in an $R$-neighborhood of $\alpha$ by the assumption. Then we can take two points $z_i\in \beta_i$ ($i=1,2$) far from $N_{R}(X)$ such that  $d(z_i, \alpha)\le R$ and thus the ball centered at $z_i$ of radius $R$ misses $X$. By the contracting property, $d_{X}(B(z_i, R))\le C$ for  each ball around $z_i$ of radius $R$. By a projection argument, we obtain $$\len(p_i)\le d_{X}(\alpha)+ \sum_{i=1,2} \Large( d_{X}(\beta_i) +d_{X}(B(z_i, R)) \Large)\le 2B+3C.$$ This gives a contradiction if $D>2B+3C$ is assumed. Hence $\alpha \cap N_C(X) \ne \emptyset$ is proved. Let $\alpha_1$ be the left geodesic ray in $\alpha \setminus  N_C(X)$ so we have $d_X(\alpha_1)\le C$ and $d(\alpha_1)_-, X)\le C$. Again a projection argument shows $$d((p_i)_-, \alpha_1)\le  d_X(\alpha_1)+ d_X(\beta_1) +C\le B+2C.$$ Set $\epsilon=B+2C$ completes the proof of (3).  

The assertion (2) is an immediate consequence of (3). Indeed,  let $\alpha$ be any geodesic segment from $X_j$ to $X_i$ for $i<j$. By (3), we have $\alpha\cap B((p_i)_+, \epsilon)\ne \emptyset$, so $\pi_{X_i}(X_j)\subset B((p_i)_+,\epsilon)$.
\end{proof}

\subsection{Exponential negiligibility of barrier-free elements} Recall that the notion of a statistically convex-cocompact action is given in Definition \ref{StatConvex}, so that 
$$\e {\mathcal O_{M_1, M_2}} < \e G<\infty$$
for some constants $M_1, M_2>0$. Here, $\mathcal O_{M_1, M_2}$ is  the set  of elements $g\in G$ such that there exists some geodesic $\gamma$ between $\gamma_-\in B(o, M_2)$ and $\gamma_+\in B(go, M_2)$ with the property that the interior of $\gamma$ lies outside $N_{M_1}(Go)$.   In the applications under consideration, since $\mathcal O_{M_2, M_2} \subset \mathcal O_{M_1, M_2}$, we can assume that $M_1=M_2$ and from now on, denote $\mathcal O_M:=\mathcal O_{M, M}$ for an easy notation.  

The next tool in our study  is the (exponential) negiligibility of a class of barrier-free elements. A set $X$ in $G$ is called \textit{generic} if 
$$
\frac{\sharp X\cap N(o, n)}{\sharp N(o, n)}\to 1,
$$
as $n\to \infty$. It is called \textit{exponentially generic} if 
$$
\frac{\sharp N(o,n)\setminus X}{\sharp N(o, n)}\le \exp(-\epsilon n),
$$
for some $\epsilon>0$ and all $n\gg 0$.

By definition, a subset $X$ of $G$ is called \textit{growth tight} if $\e X<\e G$. If $G$ has purely exponential growth then the growth tightness of a subset  is equivalent to its exponential negiligibility.

We now recall a notion of barriers on a geodesic and an element from \cite{YANG10}. 
\begin{defn}\label{barriers}
Fix constants $\epsilon, M>0$ and a set $P$ in $G$. \begin{enumerate}
\item 
Given $\epsilon>0$ and $f\in P$, we say that a geodesic $\gamma$ contains an \textit{$(\epsilon, f)$-barrier}   if there exists  an element $t \in G$ so that 
$$
\max\{d(t\cdot o, \gamma), \; d(t\cdot fo, \gamma)\}\le \epsilon.
$$
If  there exists no  $t \in G$   so that the above inequality holds, then $\gamma$ is called \textit{$(\epsilon, f)$-barrier-free}.

 Generally, $\gamma$ is called \textit{$(\epsilon, P)$-barrier-free} if it is $(\epsilon,f)$-barrier-free for some $f\in P$. An obvious fact is that  any subsegment of $\gamma$ is also $(\epsilon, P)$-barrier-free.

\item 
An element $g\in G$ is \textit{$(\epsilon, M, P)$-barrier-free} if there exists  an $(\epsilon, P)$-barrier-free geodesic between $B(o, M)$ and $B(go, M)$.  {Denote by $\mathcal V_{\epsilon, M, P}$ the set of $(\epsilon, M, P)$-barrier-free elements in $G$.}
 
\end{enumerate}
\end{defn}

\begin{rem}
By abuse of language and for simplicity, we shall say that the contracting subset $t\cdot \ax(f)$ (or even the element $t$ when $f$ is clear in context) is an $(\epsilon, f)$-barrier. 
\end{rem}

\begin{thm}\cite[Thm. C \& Cor. 4.5]{YANG10}\label{GrowthTightThm}
Assume that a non-elementary group $G$ acts properly on a proper geodesic metric $\mathrm Y$ with a contracting element. Then for any $M\gg 0$ there exists $\epsilon>0$ such that   the following properties hold for any  element $f\in G$:

\begin{enumerate}
\item
If the action is SCC, then the barrier-free set $\mathcal V_{\epsilon, M, f}$ is exponentially negligible.
\item
If the    action has PEG, then the barrier-free set $\mathcal V_{\epsilon, M, f}$ is negligible.
\end{enumerate}

\end{thm} 
 
We remark that if the action is SCC, then it has purely exponential growth. However, the converse is not true: there exists examples of geometrically finite Kleinian groups on Hadamard manifolds with PEG actions but without the parabolic gap property. Indeed, M. Peign\'e \cite{Peigne} constructed a class of exotic Schottky groups acting geometrically finitely on a simply connected Hadamard manifold without parabolic gap property so that the corresponding Bowen-Margulis-Sullivan measure is finite. By Roblin's work \cite{Roblin}, the finiteness of  Bowen-Margulis-Sullivan measure is equivalent to the purely exponential growth of the action.



\section{Genericity properties of contracting elements}\label{Section3}

The goal of this section is to choose an exponentially generic set of contracting elements so that we can compute efficiently their stable length. We start by recalling some results from \cite{YANG11} about genericity of contracting elements.  

Throughout this section, let $M>0$ be the constant appearing in the definition of a statistically convex-cocompact action. Let $\epsilon=\epsilon(M)>0$ given by Theorem \ref{GrowthTightThm} so that the barrier-free set  $\mathcal V_{\epsilon, M, f}$ is exponentially negligible for any $f\in G$. We shall omit $M$ in $\mathcal V_{\epsilon, M, f}$ when it is clear from context.

\subsection{Genericity of contracting elements}
In \cite{YANG11}, it is  proved that contracting elements are (resp. exponentially) generic if the proper action is PEG (resp.  SCC). The proof of this result replies on the following more general result. Recall that $\mathcal V_{\epsilon, f}$ denotes the set of $(\epsilon, f)$-barrier-free elements.

\begin{thm} \cite[Theorem 4.1]{YANG11} \label{Conj2BarrierFree}
For each contracting element $f\in G$   the set of elements in $G$ conjugated into $\mathcal V_{\epsilon, f}$ is exponentially negligible for SCC actions.
\end{thm}
\begin{rem}
Parallel  to Theorem \ref{GrowthTightThm}, there is an additional statement that, when the action is PEG,   the above set in the conclusion is negligible. This might be used to generalize  \ref{mainthm} to any proper PEG action.
\end{rem}
 
With the fixed basepoint $o$ in mind, an element $g\in G$ in its conjugacy class is called \textit{minimal} if $d(o, go)\le d(o, ho)$ for any $h\in [g]$. The second named author proved in \cite[Theorem 3.1]{YANG11} that non-contracting elements admit  minimal conjugacy representatives in  $\mathcal V_{\epsilon, f^m}$ for some $m>0$. Together with Theorem \ref{Conj2BarrierFree} this implies that  non-contracting elements are exponentially negligible. As a consequence, in order to count conjugacy classes ordered by algebraic length, it suffices to consider contracting elements. Moreover, we can consider the set of elements admitting a minimal representative with an $(\epsilon, f)$-barrier. We record this observation into the following.

\begin{lem}\label{TightConjugacyBarrier}
For a SCC action, the set of non-contracting elements is  exponentially negligible. Moreover, the set of   elements which admit a minimal $(\epsilon, f)$-barrier-free representative is exponentially negligible.  
\end{lem}

\subsection{Stable axis of contracting elements} 

One of the goals of this section is to show that  for a generic set of conjugacy classes, the stable length and algebraic length differ only by a bounded amount. For that purpose, we shall introduce a notion of stable axis in order to facilitate the computation of stable length. 

Given    a contracting element $g$, the group $E(g)$ is  the maximal elementary subgroup containing $g$ in $G$, and $\ax(g)=E(g)\cdot o$ is the \textit{(coarse) axis} of $g$ (depending on the basepoint $o\in \mathrm Y$).

We now introduce a {finer} notion of axis for a contracting element, which is a \textit{stable} subset of the coarse axis $\ax(g)$ in a sense   that it belongs to the negatively curved part of $\ax(g)$. For a pseudo-Anosov element     on Teichm\"uller space this is the (possibly disconnected) subsegment of its translation axis contained over a fixed compact part of the moduli space.  
\begin{defn}[Stable axis]\label{StabAxisDefn}
Let $g$ be a contracting element. Given $R>0$, a $\langle g\rangle$-invariant subset $\tax_R(g)\subset \mathrm Y$ is called a \textit{stable $R$-axis}, if for  any point $x\in \tax_R(g)$, the ball $B(x, R)$ intersects any bi-infinite  geodesic $\alpha$ {contained} in a finite neighborhood of $\ax(g)$,  and $d(x, gx)\ge 3R$. 
\end{defn}

The terminology of a stable axis is explained {by} the following lemma.
 
\begin{lem}\label{StableAxisLength}
Assume that a contracting element $g$ admits a  stable $R$-axis $\tax_R(g)$ for some $R>0$. Then for any $x\in \tax_R(g)$, we have  $|\tau[g]-d(x, gx)|\le 2R$.
\end{lem}
\begin{proof}
Choose a reference point $x_0\in \tax_R(g)$ and a geodesic $\alpha$ within a finite Hausdorff distance of $\ax(g)$. Since $\ax(g)$ is $\langle g\rangle$-invariant, we have that for each $i\in \mathbb Z$, $g^i \alpha$ stays in a finite neighborhood of $\ax(g)$. By definition of the stable axis, we obtain that $d(g^ix_0, \alpha)\le R$.  Denoting $x_i=g^ix_0$, there exists $z_i\in \alpha$ such that $d(x_i, z_i)\le R$ and so $|d(z_i, z_{i+1})-d(x_0, gx_0)|\le 2R$ for each $i$. Since $d(x_0, gx_0)\ge 3R$ and $g$ preserves the orientation of $\alpha$, we see that $z_i$ are linearly ordered on $\alpha$. This shows that $$|d(z_0, z_n)- n d(x_0, gx_0)|\le 2nR.$$  Since $\max\{d(z_0, x_0), d(z_n, g^nx_0)\}\le R$, we have
$$\tau(g)=\lim_{n\to \infty}\frac{d(g^n x_0, x_0)}{n} = \lim_{n\to \infty}\frac{d(z_n, z_0)}{n}.$$ 
This gives  $|\tau(g) - d(x_0, gx_0)|\le 2R,$ completing the proof.
\end{proof}



Let us now relate  the  stable axis to the combinatorial axis  of   a periodic admissible path (see Def. \ref{PeriAdmDef}).  

Fix a basepoint $o\in \mathrm Y$. In the next lemma,  we  consider a contracting element $f\in G$. Let  $\mathbb X=\{g\ax(f): g\in G\}$  the collection of $C$-contracting subsets with bounded intersection for a constant $C>0$ by \cite[Lemma 2.11]{YANG10}. 

\begin{lem}\label{LengthPeriAdmPath}
For any given $\tau>0$, there exist $D_0=D_0(f, \tau), R=R(f, \tau)>0$ such that the following holds.

Consider a contracting element $g\in G\setminus E(f)$ with a periodic $(g, D, \tau)$-admissible path $\gamma =\cup_{i\in \mathbb Z} (q_i p_i)$ for some $D>D_0$. Assume that the initial endpoint $(p_0)_-$ of $p_0$ is  the basepoint $o$. Then 
\begin{enumerate}
\item
Let $\mathbb A(g)\subset \mathbb X$ be the combinatorial axis of $\gamma$. The following union 
$$
\tax(g):=\cup  \{\proj_X(Y): X\ne Y\in \mathbb A(g)\}
$$
consists of a stable $R$-axis of the element $g$. 

\item
If $g=g_0^k$ for $k\ge 1$, then $$|\tau[g_0] - \ell_o[g_0]|\le R.$$ 
\end{enumerate}
\end{lem}
\begin{proof}
(1). Since $g^i\pi_X(Y)=\pi_{g^iX}(g^iY)$ and  $\mathbb A(g)$ is $\langle g\rangle$-invariant, we have  $\tax(g)$ is $\langle g\rangle$-invariant.

Consider   a bi-infinite geodesic $\alpha$  in a finite neighborhood of $\ax(g)$. If $\epsilon>0$ is the constant  given by Proposition \ref{admissible}, then any point $x\in \tax(g)$ is $\epsilon$-close to a point $(p_j)_-$ for some $j\in \mathbb Z$ which is in turn $\epsilon$-close  to $\alpha$. Setting $R=6\epsilon$, we have $d(x,\alpha) \le 2\epsilon <R$. 

By the definition of a periodic admissible path, we have $g(p_i)_-=(p_{i+1})_-$ for every $i\in \mathbb Z$. See a schematic picture (\ref{figure0}) of a periodic admissible path. We deduce from Proposition \ref{admissible} that $(p_0)_-$ is $\epsilon$-close to $[(q_0)_-, (p_0)_+]$, so by the periodicity we have  $$|d((p_i)_-,(p_{i+1})_-) -(\len(p_0)+\len(q_0))|\le 2\epsilon$$ 
for any $i\in \mathbb Z$. If $D>3R+4\epsilon$ is assumed,   then the above inequality with $\len(p_0)>D$ implies for  any $i\in \mathbb Z$:  $$d((p_i)_-,g(p_i)_-)\ge \len(p_0)-2\epsilon\ge 3R+2\epsilon.$$ Recalling $d(x, (p_j)_-)\le \epsilon$, we have $d(x, gx)\ge 3R$ and thus $\tax(g)$ is a stable $R$-axis. 

As a consequence, the assertion (2) for $k=1$ (i.e. $g=g_0$) follows. Indeed, by Lemma \ref{StableAxisLength}, we have $$|\tau[g]-d(x, gx)|\le 2R$$ for any $x\in \tax(g)$. By the discussion above, $x$ can be chosen to be $\epsilon$-close to the initial endpoint $o=(p_0)_-$ of $p_0$.  Since $\tau[g]\le \ell_o[g]\le d(o, go)$ always holds, we obtain 
\begin{equation}\label{k1EQ}
|\tau[g]-\ell_o[g]|\le 3R.    
\end{equation}

The full generality for $k\ge 2$ of (2) shall be proved below by exhibiting a stable axis for $g_0$ which contains $\mathcal A(g)$.

(2).  Any root of a contracting element  is contracting. Thus,   $g_0$ is contracting with the same maximal elementary group $E(g_0)=E(g)$ and the axis $\ax(g_0)=\ax(g)$. 

Observe that  $g_0^i \mathbb A(g)\cap g_0^j \mathbb A(g)=\emptyset$ for $0\le i\ne j\le k$. Indeed, note that $\mathbb A(g)=\{g^k\ax(f):k\in \mathbb Z\}$. Assume for the contrary that $g_0^{i-j} g^l\ax(f)=g^m\ax(f)$ for some $m\ne l\in \mathbb Z$.  Then $g_0^{i-j+k(l-m)}$  lies in $E(f)$ and  is   a non-trivial power of $g_0$ for $|i-j|\le k-1$.  Since $E(g_0)$ is the unique maximal elementary group containing $\langle g_0\rangle$, we obtain that $E(g)=E(g_0)=E(f)$. This however contradicts to the assumption that $g\notin E(f)$. 

Hence, $\mathbb B:=\cup_{i=0}^{k-1} g_0^i\mathbb A(g)$ is  a disjoint union of $k$ copies of $\mathbb A(g)$. By the bounded intersection, there exists $B>0$ such that for any two $X\ne X'\in \mathbb B$ we have 
\begin{equation}\label{BddIntersectionEQ}
\diam{N_C(X)\cap N_C(X')}\le B.
\end{equation}
Let $\alpha$ be a geodesic in finite neighborhood of $\ax(g_0)$. Here is the consequence of the above discussion.
\begin{claim}
If $D>2B+2\epsilon$ is assumed,   the intersections of   $\alpha$ with all $X\in \mathbb B$ appear  in a linear order so that any two distinct intersection have distance at least $D-B-2\epsilon$.
\end{claim}
\begin{proof}[Proof of the claim]
Note that $\ax(g_0)=\ax(g)$ is $\langle g_0\rangle$-invariant so  $g_0^{-i}\alpha$ lies in a finite neighborhood of $\ax(g)$ for any $i\in \mathbb Z$.  We now apply Proposition \ref{admissible}.(3) to get 
$\diam{N_C(X)\cap g_0^{-i}\alpha}\ge D-2\epsilon$
for every $X\in \mathbb A(g)$. Thus, for any $X\in \mathbb B$, we have 
\begin{equation}\label{IntersectAlphaEQ}
\diam{N_C(X)\cap \alpha}\ge D-2\epsilon    
\end{equation}
The claim thus follows by combining (\ref{IntersectAlphaEQ}) and (\ref{BddIntersectionEQ}).
\end{proof}

We are ready to show  that  the following $\langle g_0\rangle $-invariant set  $$\tax(g_0):= \cup_{i=0}^{k-1} g_0^i\tax(g)$$    as the union of $k$ copies of $\tax(g)$ is a stable $R$-axis of the element $g_0$.  We noted above that  if $\alpha$ lies in a finite neighborhood of $\ax(g_0)=\ax(g)$,  then   $\alpha$ intersects        the $R$-neighborhood of   $g_0^i\tax(g)$ for each $i$, and of thus $\tax(g_0)$.  It thus remains to show $d(x, g_0x)\ge 3R$ for each $x\in\tax(g_0)$. 

For concreteness, assume that  $ g_0^{-i} x\in\tax(g)$ for some $0\le i\le k-1$. Set $\beta:=g_0^{-i} \alpha$. Then  $\beta$   lies in a finite neighborhood of $\ax(g)$. By Proposition \ref{admissible}, there exists an endpoint  of $p_j$ denoted by $x_0\in X_j$ such that   $d(g_0^{-i}x, x_0)\le \epsilon$. On the other hand, $X_j\in \mathbb A(g)$ and $g_0X_j\in \mathbb B$ so the claim above implies   $d(x_0, g_0 x_0)\ge D-2\epsilon-B$.  Thus, if $D\ge 3R+B+4\epsilon$ is assumed, then $$ d(x, g_0x)\ge d(x_0, g_0x_0) -2\epsilon \ge 3R.$$

Since $\mathcal A(g)$ is included in $\mathcal A(g_0)$ by definition, the inequality (\ref{k1EQ}) follows similarly for $g_0$. Resetting $R:=3R$ finishes the  proof of (2). 
\end{proof}

\subsection{Stable length of   generic contracting elements} 
The main result of this subsection is that a generic contracting element admits a periodic admissible path, so its stable length could be estimated by Lemma \ref{LengthPeriAdmPath}.   

\begin{prop}\label{StablePointedLength}
Fix a basepoint $o\in \mathrm Y$,  and two independent contracting elements $f_1, f_2\in G$ such that $d(o, f_io)\gg 0$. There exists a constant $R=R(o, f_1, f_2)>0$ with the following property.

Let $g\in G$ be a  {minimal} conjugacy representative in $[g]$ so that $[o, go]$ contains at least one $(\epsilon, f_1)$-barrier and one $(\epsilon, f_2)$-barrier. 
Then $$|\tau[g]-d(o, go)|\le R.$$ 
Moreover, for any $g_0^k=g$ with $k\in \mathbb Z$, we have $|\tau[g_0]-d(o, g_0o)|\le R$.
\end{prop}
\begin{proof}
Assume that there exists an $(\epsilon ,f_i)$-barrier $t_i$ for $\alpha=[o, go]$ where $i=1, 2$. Let $\mathbf b_i= t_i\cdot\ax(f_i)$. Since $f_1, f_2$ are independent, we have that $\mathbf b_1$ and $\mathbf b_2$ have bounded intersection: for given $\epsilon>0$ there exists $B=B(\epsilon)>0$ such that $$\diam{N_\epsilon(\mathbf b_1)\cap N_\epsilon(\mathbf b_2)}\le B.$$ 
The proof is based on the following.
\begin{claim}
There exists $\mathbf b\in \{\mathbf b_1, \mathbf b_2\}$ with the following property:
 
Denote by $x, y$ the entry and exit points of the geodesic $\alpha$ in $N_\epsilon(\mathbf b)$. We have
\begin{equation}\label{MinDistEQ}
d(x, y)+d(y, gx)\ge d(o, go)-2\epsilon,
\end{equation}
\begin{equation}\label{ShortRightIntersectEQ}
\diam{N_\epsilon(\mathbf b)\cap g[o, x]_\alpha}\le B+4\epsilon,
\end{equation}
\begin{equation}\label{ShortLeftIntersectEQ}
\diam{N_\epsilon(\mathbf b)\cap g^{-1}[y, go]_\alpha}\le B+4\epsilon.
\end{equation}
\end{claim}

We will first complete the proof of the proposition assuming the \textbf{Claim}. 
To that end, we construct the concatenated path as follows: 
\begin{equation}\label{PeriodicAdmEQ}
\gamma=\cup_{i\in \mathbb Z} g^i([x, y]_\alpha \cdot [y, gx]),
\end{equation}
where $[x, y]_\alpha$ is the subsegment of $\alpha$ between $x$ and $y$.

Denote $D=\max\{d(o, f_io): i=1, 2\}-2\epsilon>0$ and $\sigma=9\epsilon+B$. We now verify that   $\gamma$ is a periodic $(g, D, \sigma)$-admissible path with associated contracting subsets $\mathbb A(\gamma)=\{g^i N_\epsilon(\mathbf b): i\in \mathbb Z\}$ of bounded intersection. 

Since $\mathbb X:=\{g\ax(f_1), g\ax(f_2): g\in G\}$ has bounded intersection and $d(x, y)\ge D$, according to Definition \ref{PeriAdmDef} of admissible paths, it thus remains to verify the following condition (\textbf{BP}):  
\begin{center}
$[y, gx]$ has $\sigma$-bounded projection to $N_\epsilon(\mathbf b)$ and $N_\epsilon(g\mathbf b)$. 
\end{center}
We only prove it for $N_\epsilon(\mathbf b)$; the   case for $N_\epsilon(g\mathbf b)$ is symmetric. 

Note that $y$ is the exit point of $\alpha$ in $N_\epsilon(\mathbf b)$ and since $\epsilon>C$, we have $[y, go]\cap N_C(\mathbf b)=\emptyset$.  The contracting property then implies $d_{\mathbf b}([y, go])\le C$. By \cite[Lemma 6.1]{YANG11}, we obtain that $d_{\mathbf b}(g[o, x]_\alpha) \le \diam{N_\epsilon(\mathbf b)\cap g[o, x]_\alpha}+4C$. By  (\ref{ShortRightIntersectEQ}) from the \textbf{Claim},
$$
d_{\mathbf b}([y, gx]) \le d_{\mathbf b}([y, go]) +d_{\mathbf b}(g[o, x]_\alpha)\le 5C+B+4\epsilon \le \sigma.
$$

Hence,  $\gamma$ is a periodic $(g, D, \sigma)$-admissible path. Choose $d(o, f_io)\gg 0$ for $i=1, 2$ such that the constant $D>D_0$ satisfies Proposition \ref{admissible}.  By Lemma \ref{LengthPeriAdmPath},  there exists $R>0$ such that $$|\tau[g]-d(x, y)-d(y,gx)|\le R.$$ With (\ref{MinDistEQ}), the   estimates on stable length of $g$ follows and the proposition is proved.
 
\begin{figure}[htb] 
\centering \scalebox{0.6}{
\includegraphics{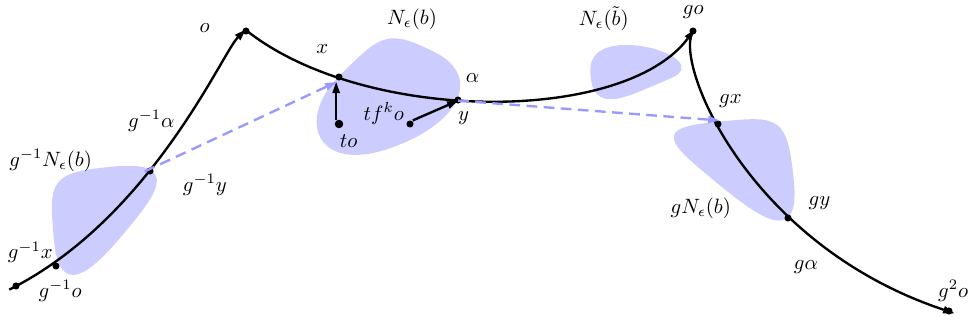} 
} \caption{Capping the angles by  dotted geodesics gives the periodic admissible path $\gamma$} \label{figure1}
\end{figure}

Therefore, it remains to prove the above \textbf{Claim}. 
\begin{proof}[Proof of the Claim]

We shall first prove  (\ref{MinDistEQ}) for any  $\mathbf b\in \{\mathbf b_1, \mathbf b_2\}$. Denote by $x, y$ the corresponding entry and exit points of the geodesic $\alpha$ in $N_\epsilon(\mathbf b)$.  Since $\mathbf {b}= t\ax(f)$ for some $f\in \{f_1, f_2\}$ and $t\in G$,  there exists  $0\ne k\in \mathbb Z$ such that $d(y, tf^ko)\le \epsilon$.

Let us look at the triangle $\Delta(y, gx, go)$. The minimality of $g$ in $[g]$ implies the second inequality:
\begin{equation}\label{dygyEQ}
d(y, gy)\ge d(tf^ko, gtf^ko)-2\epsilon \ge d(o, go)-2\epsilon.
\end{equation}
Similarly, we have $d(x, gx)\ge d(o, go)-2\epsilon.$

Note that $d(o, go)= d(o, x)+d(x, y)+d(y, go)$ and $d(y, gy) \le d(y, gx) +d(gx, gy).$ From (\ref{dygyEQ}), we  infer that 
\begin{equation}\label{reverseTIQ}
d(y, go)+d(go, gx)\le d(y, gx)+2\epsilon 
\end{equation}
Together $d(o, go)=d(o, x)+d(x,y)+d(y, go)$, we obtain (\ref{MinDistEQ}) for any $\mathbf b\in \{\mathbf b_1, \mathbf b_2\}$.

We are next going to prove (\ref{ShortRightIntersectEQ}) only; (\ref{ShortLeftIntersectEQ}) is similar.  By way of contradiction,  assume that for any $\mathbf b\in \{\mathbf b_1, \mathbf b_2\}$ the following holds:
\begin{equation}\label{LongIntersectEQ}
\diam{N_\epsilon(\mathbf b)\cap g[o, x]_\alpha}> B+6\epsilon 
\end{equation}

By assumption,  $\alpha=[o, go]$ contains at least two distinct barriers. Up to exchanging notations, we choose   $\mathbf b \neq \mathbf {\tilde b} \in \{\mathbf b_1,\mathbf b_2\}$ such that $\mathbf {\tilde b}$ is on the right side of $\mathbf { b}$ (See Figure \ref{figure1}).  As above, let $x,y$ be the entry and exit points of $\alpha$ in $N_\epsilon(\mathbf b)$. We shall need the following estimates to arrive at  a contradiction:
 \begin{equation}\label{dygoEQ}
d(y, go)\le  4\epsilon.
\end{equation}

Indeed,    let $w$ be the entry point of $[go, gx]$ in $N_\epsilon(\mathbf b)$. Since    $\epsilon>C$ is assumed,    the contracting property implies $d(y, w)\le d_{\mathbf b}([y, go]_\alpha)+d_{\mathbf b}([go, w]_{g\alpha})\le 2C$.  Hence, for $w\in [go, gx],$ we have $$d(y, gx)\le d(w, gx)+d(y, w) \le d(w, gx)+2C\le d(go,gx)+2C.$$ By (\ref{reverseTIQ}), we then obtain that  $$d(y, go)\le d(y, gx)-d(go, gx)+2\epsilon \le 2C+2\epsilon\le 4\epsilon.$$ 
The inequality in (\ref{dygoEQ}) is then proved.

Recall that $\mathbf {\tilde b}$ is chosen on the right side of $\mathbf {b}$. By $B$-bounded intersection  of  $N_\epsilon(\mathbf {\tilde b})$ and $N_\epsilon(\mathbf b)$, we deduce from  (\ref{LongIntersectEQ}) that the corresponding exit  points of   $\alpha$ in $N_\epsilon(\mathbf b)$ and $N_\epsilon(\mathbf {\tilde b})$ have distance strictly greater than $4\epsilon$. Hence, $d(y, go)>4\epsilon$     contradicts (\ref{dygoEQ}). 

Therefore, there exists $\mathbf b  \in \{\mathbf b_1,\mathbf b_2\}$ such that  (\ref{ShortRightIntersectEQ}) and (\ref{ShortLeftIntersectEQ}) are true. The \textbf{Claim} is proved.
\end{proof}
The proof of the proposition is complete. 
\end{proof}

We now record the main consequence of Proposition \ref{StablePointedLength}. 

\begin{cor}\label{GenericSetTwoLengths}
 
There exist a number $B>0$ and an exponentially generic set $\mathcal G$ of elements such that for each $g\in \mathcal G$,   $0\le \len_o[g]-\tau[g]\le B$. Moreover, for any $g_0^k=g$ with $k\in \mathbb Z$, we have $|\tau[g_0]-\len_o[g_0]|\le B$.
\end{cor}
\begin{proof}
Only the exponential genericity of elements   in Proposition \ref{StablePointedLength} needs to be verified, which follows from Theorem \ref{Conj2BarrierFree}.
\end{proof}
\begin{rem}
Theorem \ref{Conj2BarrierFree} is only invoked here in this section,  and the further results (Lemma \ref{fractalbarrier}, Theorem \ref{LinearDrift} below) about genericity do not rely on it.
\end{rem}

\subsection{Growth tightness of fractional barrier-free elements}\label{SSfracbarrier}
The goal of the remaining subsections (independent of the proof of  \ref{mainthm})  is Theorem \ref{LinearDrift}  which shall be used in Theorem \ref{Teichprincipal} about generic Teichm\"{u}ller geodesics.      

Let $\overrightarrow \theta(n)=(\theta_1(n), \theta_2(n)): \mathbb R_{\ge 0} \to [0, 1]\times [0,1]$ be a  function such that $\theta_1(n)\le  \theta_2(n).$
A \textit{$\overrightarrow \theta$-interval} of a geodesic {segment} $\alpha$, denoted by $\alpha_{\overrightarrow \theta}$,  is the closed subsegment so that the initial endpoint  has a   distance $\theta_1(n)  n$ to $\alpha_-$ and the terminal endpoint has a distance $\theta_2(n)  n$ to $\alpha_+$ where $n:=\len(\alpha)$.   In most cases, we choose a constant function $\overrightarrow\theta(n)=(\theta_1, \theta_2)$ for $0\le \theta_1<\theta_2\le 1$.

\begin{lem}\label{fractalbarrier}
Assume that the action $G \curvearrowright\mathrm Y$ is SCC. Let   $\overrightarrow \theta(n):=(\theta_1(n), \theta_2(n)): \mathbb R_{\ge 0} \to [0, 1]^2$ be a   function such that $\theta(n)=\theta_2(n)-  \theta_1(n)\ge n^{-a}$ for some $a\in (0, 1)$.

{Let $f$ be a contracting element.} Then the set of elements   $g\in G$ for which the interval  $[o,go]_{\overrightarrow \theta}$ does not contain an $(\epsilon, f)$-barrier is negligible.  

Moreover, when $\displaystyle \liminf_{n\ge 1}\theta(n)>0$, the above set is exponentially negligible. 
\end{lem}

\begin{rem}
The    SCC assumption on actions is crucial to obtain the growth tightness in the ``moreover'' statement. 
\end{rem}

In its  proof, we frequently use the following criterion {for a set to be negligible}.
\begin{lem}\label{ThreeDecompTight}
Assume that the action $G \curvearrowright\mathrm Y$ has purely exponential growth. Let $D>0$   and $\theta : \mathbb R_{\ge 0} \to (0, 1]$ such that $\theta(n)\ge  n^{-a}$ for some $1>a>0$.
For    a growth tight subset $Z\subset G$, the set of elements  $g\in G$     satisfying   the following two properties:
\begin{enumerate}
\item
$g=g_1g_2g_3$ can be written as a product of three elements  such that  $|d(o, go)-\sum_{i=1,2,3} d(o, g_io)|\le D$ and 
\item
one of the three $g_i$'s  belongs to $Z$ and has length bigger than $\theta(d(o, go))   \cdot d(o, go)$,
\end{enumerate} 
is negligible. Moreover, when $\displaystyle \liminf_{n\ge 1}\theta(n)>0$, the above set is exponentially negligible. 
\end{lem}
\begin{proof}
Let us first consider the set of elements $g=g_1g_2g_3$, where $g_2$ satisfies the second property. Denote $n:=d(o, go)$ and $\theta_n:=\theta(n)\ge n^{-a}$ for some $a\in (0,1)$.  The number of these elements is upper bounded by 
\begin{equation}\label{SumEQ}
\begin{array}{rl}
 \sum\limits_{\substack{0\le k+l\le n+D  \\ n\ge l\ge n\theta_n  }} \sharp N(o, k)\cdot \sharp \big (N(o, l) \cap Z\big)\cdot \sharp N(o, n+D-k-l).
\end{array}
\end{equation}

Since the action has PEG, we have $$\sharp N(o, i) \asymp \exp(\e G i)$$ for $i\ge 0$. Since  the set $Z$ is growth tight, there exists $0<\omega_1<\e G$ such that $$\sharp(N(o, l)\cap Z)\prec \exp(\omega_1 l).$$ 
For  $\epsilon:=\e G-\omega_1>0$, each summand in (\ref{SumEQ}) takes proportion of $N(o, n)$ at most $\prec \exp(-\epsilon \cdot  n\theta_n  ) \prec \exp(- {\epsilon}{n^{1-a}} )$. Since there are at most $n^2$ summands and $n^2 \exp(-\epsilon n^{1-a})\to 0$, we obtain that  the set of elements $g$ with this property is negligible.   When $g_1$ or $g_3$ satisfies the second property, an even simpler proof shows that the {corresponding sets} are negligible as well. 

If $\theta_n$ is uniformly away from $0$, then the above computation shows that the set under consideration is exponentially negligible. Thus the result is proved.
\end{proof}

We are ready to prove Lemma \ref{fractalbarrier}. 
\begin{proof}[Proof of Lemma \ref{fractalbarrier}]
Denote $\beta=\alpha_{\overrightarrow \theta}$ the $\overrightarrow \theta$-interval of $\alpha:=[o, go]$, and  for simplicity,   $\theta_2=\theta_2(n), \theta_1=\theta_1(n)$ where $n=d(o, go)$. Let $g\in G$ be an element so that $\beta$ does not contain an   $(\epsilon, f)$-barrier.    The proof proceeds by showing that $g$ belongs to a finite union of negligible sets.  
 
 First of all, we can assume that the entire geodesic $\alpha$ contains an $(\epsilon, f)$-barrier. Otherwise, $g$ belongs to the barrier-free set $\mathcal V_{\epsilon, f}$, which is exponentially negligible by Theorem \ref{GrowthTightThm}. 

Let $t\in G$ be any $(\epsilon, f)$-barrier such that $d(to, \alpha), d(tfo, \alpha)\le \epsilon.$ Let $x, y\in \alpha$ such that $d(to, x), d(tfo, y)\le \epsilon$.

\begin{claim}
The segment $[x,y]_\alpha$ intersects  $\beta$ in a diameter at most $\len(\beta)/3$.
\end{claim}
\begin{proof}[Proof of the Claim]
Note $$
d(o,go) +4\epsilon \ge  d(o, to) +d(to, tfo)+d(tfo, go).
$$
Consider the growth tight set $Z:=E(f)$,     $D:=4\epsilon$ and $\theta:=(\theta_2-\theta_1)/3$. If $\len([x,y]_\alpha\cap \beta)\ge \len(\beta)/3$, then the element $g=t\cdot f \cdot (tf)^{-1} g$ can be written as a product of three elements satisfying the properties  of Lemma \ref{ThreeDecompTight}. Hence, the set of  these elements $g$ is negligible. So we can assume that the subpath $[x, y]_\alpha$ intersects  $\beta$ in a diameter at most $\len(\beta)/3$. 
\end{proof}

 
By the Claim,  $[x, y]_\alpha$ does not contain the middle point of $\beta$, so any barrier of $\alpha$ stays either on the left side or on the right side of $\beta$.

If one side, say the left side,  of $\beta$ does not contain an $(\epsilon, f)$-barrier, then the right side of $\beta$ must contain at least one $(\epsilon, f)$-barrier since $\alpha$ contains one.  Consider the left-most $(\epsilon, f)$-barrier $t_2\in G$ with the entry point $w$ of $\alpha$ in $N_\epsilon(t_2 \ax(f))$. This implies that $[o, w]_\alpha$ is $(\epsilon, f)$-barrier-free and  $\len([o, w]_\alpha)\ge 2/3\cdot  \len(\beta)\ge 2n\theta$. Recalling that $\ax(f)=E(f)o$, we choose an element $\tilde t_2\in t_2E(f)$ such that $d(w, \tilde t_2 o)\le \epsilon$. We then write $g=\tilde t_2\cdot (\tilde t_2^{-1}g)$ as a product of two elements with $\tilde t_2$ being $(\epsilon, f)$-barrier-free. Apply Lemma \ref{ThreeDecompTight}  with   the growth right $Z:=\mathcal V_{\epsilon, f}$, $D=2\epsilon$ and $2\theta$. We thus see that, in this case, $g$ belongs to a negligible set as well.
 
We now turn to the case that  each side of $\beta$ contains an $(\epsilon, f)$-barrier. Let $t_1\in G$ be the right-most barrier on the left side of $\beta$. Thus,  if we denote by $z\in \alpha$   the exit point in $N_\epsilon(t_1\cdot \ax(f))$, we have $d(z, \tilde t_1o)\le \epsilon$ for some $\tilde t_1\in t_1 E(f)$. Similarly, for the left-most barrier $t_2$ on the right side of $\beta$, we have $d(w, \tilde t_2o)\le \epsilon$ for   the entry point $w\in \alpha$ in $N(t_2\cdot \ax(f))$ and for some $\tilde t_2\in t_2 E(f)$. 

Recall that each barrier intersects $\beta$ in a segment of length less than $\len(\beta)/3$, and $\beta$  contains no   barrier by assumption. We conclude that $[z, w]_\alpha$ is $(\epsilon, f)$-barrier-free, and therefore so is the element $t_1^{-1} t_2$.  Noting that  $\len([z, w]_\alpha)\ge \len(\beta)/3\ge n\theta$ and writing $$g=\tilde t_1\cdot (\tilde t_1^{-1}\tilde t_2)\cdot (\tilde t_2^{-1}\tilde t_1 g)$$ as a product of three elements, {the set of} such elements $g$ is negligible by using Lemma \ref{ThreeDecompTight} again.

In summary, assuming that $\beta$  does not contain an   $(\epsilon, f)$-barrier, we obtain that the set of these elements $g$ satisfying (1) and (2) is negligible. 

When $\liminf_{n\ge 1}\theta(n)>0$, Lemma \ref{ThreeDecompTight}  shows the above sets are exponentially negligible. Hence, the ``moreover'' statement is proved along the same lines.
\end{proof}

\subsection{Generic elements with linear stable length}

In this subsection, we   prove that generic contracting elements  have their stable length   (at most) sub-linearly   close to the radius and their axis sublinearly near to the basepoint.

Similar to that of Proposition \ref{StablePointedLength},    the proof consists  in constructing periodic admissible paths but which is technically much simpler.  We emphasize that the arguments here do not use  Theorem \ref{Conj2BarrierFree}, so this gives a new proof of the main result of \cite{YANG11} that exponentially generic elements are contracting.

\begin{thm}\label{LinearDrift}
Assume that $G$ admits a SCC action with a contracting element on a proper geodesic metric space $(\mathrm Y, d)$. There exist   constants $\epsilon=\epsilon(f), R=R(f)>0$ for  any contracting element $f\in G$   such that  the following holds. 

Let $\theta_i:\mathbb R_{\ge 0}\to (0, 1]$ be such that $\theta_i(n)\ge n^{-a}$ and $a\in (0,1)$ where $i=1,2$. Then  for  any integer $m>0$,  the set of elements $g$ with $n=d(o, go)$ {satisfying the following properties is generic}:
\begin{enumerate}
\item
$n\ge \tau [g]  \ge (1-\theta_1(n)) n$,
\item
$ d(o, \tax_R(g)) \le n \theta_2(n)$,
\item
any  bi-infinite geodesic in a finite neighborhood of $\ax(g)$ contains an $(\epsilon ,f^m)$-barrier.
\end{enumerate}

Moreover, if $\liminf_{n\ge 1}\theta_i(n)>0$ for $i=1, 2$, then the above set is exponentially generic.
\end{thm}

By definition of a stable axis, we could replace $\tax_R(g)$ in the above property (2) with any bi-infinite geodesic   in a finite neighborhood of $\ax(f)$.  

\begin{proof}
We shall assume $\theta_1(\cdot)=\theta_2(\cdot)$ for ease of exposition. The general case   follows by taking  the intersection of two (exponentially) generic sets.

Fix any  $\theta:\mathbb R_{\ge 0}\to (0, 1/2]$ such that $\theta(n)\ge n^{-a}$ for some $a\in (0,1)$. For given $m>0$, we shall choose a big integer $k=k(m)>0$ determined below (at the last paragraph). Let $\mathcal G$ be the set of elements $g\in G$ satisfying the following two conditions:
\begin{enumerate}
\item
 the $[\frac{\theta_n}{4}, \frac{\theta_n}{2}]$-interval of $\alpha:=[o, go]$ contains an  $(\epsilon, f^k)$-barrier $t_1\in G$,
 \item
 the $[1-\frac{\theta_n}{2}, 1-\frac{\theta_n}{4}]$-interval of   $\alpha$ contains an   $(\epsilon, f^k)$-barrier $t_2\in G$,
 \end{enumerate}
  where    $\theta_n:=\theta(n)$ and  $n=d(o, go)$. 

By Lemma \ref{fractalbarrier}, the set $\mathcal G$ is generic, and when $\liminf_{n\ge 1}\theta(n)>0$, this set is exponentially generic. 

Denote $\mathbf b_1=t_1\ax(f)$ and $\mathbf b_2=t_2\ax(f)$. Up to ignoring a growth tight set, we can first assume that $\mathbf b_2\cap N_C(g\alpha)=\emptyset$ and $\mathbf b_1\cap N_C(g^{-1}\alpha)=\emptyset$. Indeed, we shall prove the set of elements $g\in G$  satisfying $\mathbf b_2\cap N_C(g\alpha)\ne \emptyset$ is growth tight; the other   possibility is analogous. 

\begin{figure}[htb] 
\centering \scalebox{0.8}{
\includegraphics{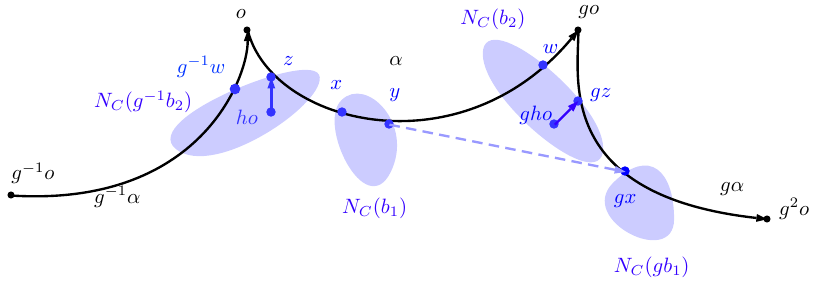} 
} \caption{The relative position of the barriers $g^{-1}\mathbf b_2, \mathbf b_1, \mathbf b_2$ and $g\mathbf b_1$.} \label{figure2}
\end{figure}

Let $z$ be the entry point of $\alpha$ in $N_C(g^{-1}\mathbf b_2)$, so $gz$ is the entry point of $g\alpha$ in $N_C(\mathbf b_2)$. Let $w$ be the exit point of $\alpha$ in $N_C(\mathbf b_2)$. If two geodesics issuing from the same point intersect the $C$-neighborhood of a $C$-contracting set, we can deduce from the contracting property that  their entry points are bounded above by a distance $4C$. Thus, $d(gz, w)\le 4C$. 

Let us choose $h\in G$ such that $d(ho, z)\le C$. Then $|d(gho, ho)-d(z, w)|\le d(gho, w)+d(z, ho)\le C+d(gho, gz)+d(gz, w)\le 6C$. Setting $\hat g: =h^{-1} g h$, we see that the element $g=h\hat g h^{-1}$ is an almost geodesic decomposition:
\begin{equation}\label{almostgeodecom}
d(o, go)\simeq_{24C} 2d(o, ho)+d(o, \hat go).
\end{equation}   
By the condition (1) as above, we see that $d(o, ho)\ge d(o, z)-d(z, ho)\ge \theta_n/4 \cdot d(o, go) - C$. By Lemma \ref{ThreeDecompTight}, we deduce from (\ref{almostgeodecom})  that the set of   $g$ with $\mathbf b_2\cap N_C(g\alpha)\ne \emptyset$ is   growth tight.
 
From now on, let us assume 
that $\mathbf b_2\cap N_C(g\alpha)=\emptyset$ and $\mathbf b_1\cap N_C(g^{-1}\alpha)=\emptyset$. We shall construct a periodic admissible path for the element $g$. 
 
Let $x, y$ be the entry and exit points of $\alpha$ in $N_C(\mathbf b_1)$ respectively.
We now prove that a path $\gamma$ constructed as follows is a $(g, D, C)$-admissible path  
$$
\gamma=\cup_{i\in \mathbb Z} g^i([x, y]_\alpha \cdot [y, gx]),
$$
where $g^i[x, y]_\alpha$ are associated with contracting sets $g^i N_C(\mathbf b_1)$ with bounded intersection. By the definition of admissible path, it remains to show that $[y, gx]$ has a bounded $C$-projection to $N_C(\mathbf b_1)$  and $N_C(g\mathbf b_1)$. Since  $[y, gx]$ is disjoint from them, the projection bounded by $C$ follows by the contracting property. Thus,  $\gamma$ is a periodic  $(g, D, C)$-admissible path.

We now verify the first two assertions. Let $R=R(f,C)$ be the constant given by Lemma \ref{LengthPeriAdmPath}. By the construction of $\gamma$, the collection $\mathbb X$ of contracting sets $g^i N_C(\mathbf b_1)$ for $i\in\mathbb Z$  is a combinatorial axis of $g$. By Lemma \ref{LengthPeriAdmPath}, the stable axis $\tax_R(g)$ is given by the union of mutual projections of elements in $\mathbb X$. By Proposition \ref{admissible}, the point $x\in\gamma$  is $\epsilon$-close to some point in $\tax_R(g)$ and thus $$d(o, \tax_R(g)) \le  d(o,x)+\epsilon \le n \theta_n+\epsilon.$$ 
From (\ref{k1EQ}) in the proof of Lemma \ref{LengthPeriAdmPath}, we get $$
\tau[g]\ge d(x, gx)-3R\ge d(x,go)-d(o,x)\ge  (1-\theta_n)n-3R$$ 

At last, we need to verify the assertion (3) that every geodesic $\beta$ in $\ax(g)$ contains an $(\epsilon, f^m)$-barrier for   given $m>0$. By Proposition \ref{admissible}, we obtain that $d(x, \alpha), d(y, \alpha)\le \epsilon.$ Since $x, y$ are the entry and exit points of $\alpha$ in the $(\epsilon, f^k)$-barrier $\mathbf b_1$, we have $d(x, y)\ge d(o, f^ko)-2\epsilon$. Consider $\ax(f)$ as a $C$-contracting quasi-geodesic path. By Proposition \ref{Contractions}.(1), we can choose $k$ as a large multiple of $m$ such that there exist a constant $\tilde \epsilon=\tilde \epsilon(C)>0$ and  a subpath  from $\ax(f)$ labeled by $f^m$ two endpoints in the $\tilde \epsilon$-neighborhood of $[x, y]_\alpha$. By definition of barriers,   $\beta$ contains an $(\tilde \epsilon, f^m)$-barrier. Setting  $\epsilon=\max\{\tilde \epsilon,  \epsilon\}$, the assertion (3) follows.  
\end{proof}

\section{Lower bound on   conjugacy classes}\label{Section4}
Most of this section assumes only that $G\curvearrowright \mathrm Y$ is a properly discontinuous action on a proper geodesic metric space with a contracting element,
 At the end of the section, we derive the lower bound in \ref{mainthm}   under the additional assumption action has purely exponential growth. 

We first recall a result from \cite{YANG10} which holds for any proper action. This could be thought of as an analogue to  an orbit closing lemma in Anosov flows. For any $\Delta>0$, we define $A(o, n, \Delta)=\{g\in G: |d(o, go)-n|\le \Delta\}$.

\begin{lem}\cite[Lemma 2.19]{YANG10}\label{ClosingLemma}
There exist  a set $F$ of three contracting elements   and constants $0<\theta<1, \sigma, D, \Delta>0$ with the following property. For each $n>0$, there exist a subset $T$ of $A(o, n, \Delta)$ and an element $\tilde f\in F$  such that 
\begin{enumerate}
\item  
$\sharp T \ge \theta\cdot \sharp A(o, n, \Delta)$,  
\item
for all but finitely many $f\in E(\tilde f)$, the map $$g\in T\mapsto  fg\in  f\cdot T$$ is injective.
\item
each $fg\in fT$ is a contracting element so that the path $$\gamma(fg):=\cup_{i\in \mathbb Z} {(fg)}^i [o,fo][fo, fgo]$$ is a periodic $(fg, D, \sigma)$-admissible path. 
\item
for any $g\ne g'\in T$, two paths $\gamma(fg)$ and $\gamma(fg')$ have infinite Hausdorff distance. 
\end{enumerate}
\end{lem}
\begin{proof}[Sketch of proof]
By \cite[Lemma 2.19]{YANG10}, the assertions (2-4) hold for any  $R$-separated set $T$ of $A(o, n, \Delta)$ for a sufficiently large $R>0$. The proof was to verify that the labeled path $\gamma(fg)$ by $fg$ as above is $(D, \sigma)$-admissible, i.e. a periodic $(fg, D, \sigma)$-admissible path in the terminology of this paper.   

For $g\ne g'\in T$, by Proposition \ref{admissible}, the sufficient  $R$-separation implies that if $\gamma(fg)$ and $\gamma(fg')$ has finite Hausdorff distance, then $g=g'$. So the injectivity of the above maps follows.  Choosing a maximal net $T$ then takes the major proportion of $A(o, n, \Delta)$ as is requested by the assertion (1). The quantifier ``for all but finitely many'' thus follows for any $f$ with $d(o, fo)>D$ by Proposition \ref{admissible}. 
\end{proof}

Fix a  choice $f\in E(\tilde f)$ satisfying the second statement of Lemma \ref{ClosingLemma}. {Note that $\ax(f)=\ax(\tilde f)$ since $E(f)=E(\tilde f)$.} Then the set $\tilde T:= f\cdot T$ consists of contracting elements. 

\begin{lem}\label{PointLengthT}
There exists a constant $R=R(f, \sigma, D,\Delta)>0$ such that for each element $fg\in \tilde T$, we have $|\tau[fg]-n|\le R$ and $|\len_o[fg]-n|\le R$. 
\end{lem}
\begin{proof}
By Lemma \ref{ClosingLemma}.(3), the element $fg$ admits a periodic $(fg, D, \sigma)$-admissible path $\gamma$. Thus, the conclusion follows from  Lemma \ref{LengthPeriAdmPath}.
\end{proof}

The following result is from \cite[Lemma 7.2]{CoK1}, and a proof is given for completeness.
\begin{lem}\label{UniformityfromProper} 
Given a compact set $K\subset \mathrm Y$, there exists an integer $C=C(K)>0$ such that  for  any geodesic segment $\alpha$, we have $$\sharp \{g\in G: g\alpha\cap K\ne \emptyset\}\le C \cdot\len(\alpha).$$

\end{lem}
\begin{proof}
We subdivide the geodesic $\alpha$ into a maximal number of segments $\alpha_i$ $(0\le i \le \ceil {\len(\alpha)})$ with length at most 1. Set $C:= \sharp \{g\in G: gK\cap N_1(K)\ne \emptyset\}$. Observe that  for each $\alpha_i$, {we have} $$\sharp\{g\in G: g\alpha_i\cap K\ne\emptyset\}\le C.$$ Indeed, fix an element $g_0$  so that $\alpha_i\cap g_0K\ne\emptyset$. Thus for any element $g$ from the left-hand set, we have $gg_0 K\cap N_1(K)\ne \emptyset$ and the above inequality follows. The conclusion thus  follows.
\end{proof}

Denote by $[\tilde T]$ the set of conjugacy classes $[fg]$ where $fg\in \tilde T$. Recall that $\tilde T=fT$, where $T\subset A(o,n, \Delta)$. We obtain the following lower bound.
\begin{lem}\label{NumElemInConj}
There exists a   constant $L=L(f,o, \Delta)>0$ such that each  conjugacy class $[fg]\in [\tilde T]$ contains at most $L n$ elements in $\tilde T$.
\end{lem}
\begin{proof}
Set $A:=[fg]\cap \tilde T$ and fix a choice $fg\in A$.  By Lemma \ref{ClosingLemma},  $\gamma:=\gamma(fg)$ is a periodic $(fg, D, \tau)$-admissible path  associated with contracting sets $(fg)^i\ax(f)$ where $i\in \mathbb Z$. It is clear that  $\langle fg\rangle$ acts on $\gamma$ with a fundamental domain $S:=[o, fo]  [fo, fgo]$.
 
Similarly, for any $fg'\in A\setminus \{fg\}$,  we have a periodic $(fg', D, \tau)$-admissible path   $\gamma'$ and $S':=[o, fo]  [fo, fg'o]$ a  fundamental domain for $\langle fg'\rangle$ on $\gamma'$. 

Since $h' fg'h'^{-1}=fg$ for  some $h'\in G$, it follows that $h'\gamma'$ has a finite Hausdorff distance to $\gamma$. By Proposition \ref{admissible}, there exists a   constant $\epsilon>0$ such that $\gamma$ intersects the $\epsilon$-neighborhood of every segment $(fg')^i[o,fo]$ of $h'\gamma'$. Thus,
\begin{equation}\label{intersecEQ}
N_\epsilon(S) \cap h'(fg')^i[o, fo]\ne\emptyset.
\end{equation}

\begin{figure}[htb] 
\centering \scalebox{0.8}{
\includegraphics{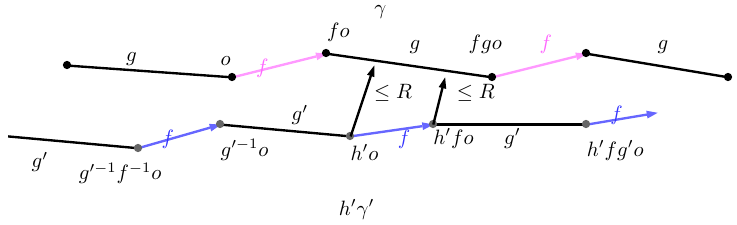} 
} \caption{Choice of the conjugator $h'$: the relative position of $h'\gamma'$ to $\gamma$.} \label{figure3}
\end{figure}

Up to composing $h'$ on the right with  elements from $\langle fg'\rangle$, we can assume that  $h'\cdot S'\cap N_\epsilon(S)\ne \emptyset$. Setting  $R=\epsilon+d(o, fo)$, we have $h' N_R([o, fo])\cap [o, go] \ne \emptyset$. 

Consider the following $$
Z:=\{h'\in G: h'K\cap [o,go]\ne\emptyset\}
$$ for the compact set $K:=N_{R}([o, fo])$.  By Lemma \ref{UniformityfromProper}, there exists $C=C(K)>0$ such that $$\sharp Z\le C \cdot d(o, go) \le  C  \cdot (n+\Delta)$$ for $g\in A(o, n, \Delta)$. Hence,     any $fg'\in A\setminus \{fg\}$ is conjugated by some $h'\in Z$  to the fixed $fg$.

Furthermore,   observe that no two $fg_1\ne fg_2\in A\setminus fg$  admit    $h(fg_i)h^{-1}=fg$ for a common conjugator  $h\in Z$. If not, the admissible paths $\gamma(fg_1)$ and $\gamma(fg_2)$ have finite Hausdorff distance. This contradicts to Lemma \ref{ClosingLemma}.(4).

In conclusion, for every $fg\in \tilde T$, the set $A=[fg]\cap \tilde T$ contains at most $\sharp A\le C(n+\Delta)+1\le Ln$ elements for some $L>0$. The constant $L$ depends only on $\Delta$, the basepoint $o$ and the choice of $f$.
\end{proof}

When the action has purely exponential growth, we obtain the lower bound for all   conjugacy classes in \ref{mainthm}.
\begin{cor}[Lower bound]\label{LbdConjGrowth}
Assume that $G\curvearrowright \mathrm Y$ is a proper action with a contracting element  having purely exponential growth on a proper geodesic metric space.  Then there exists a constant $\theta >0$ such that
$$
\sharp \mathcal C(o, n)\ge  \frac{\exp(\e Gn)}{\theta n},
$$
and $$\displaystyle\sharp \big(\mathcal C(n)\cap \mathcal C(o, n)\big)  \ge  \frac{\exp(\e Gn)}{\theta n}.$$
\end{cor} 
\begin{proof}
By Lemma  \ref{PointLengthT}, we have $[\tilde T] \subset  \mathcal C(n+R)\cap \mathcal C(o, n+R)$. Combining Lemmas \ref{ClosingLemma} and  \ref{NumElemInConj}, we obtain   $$\sharp  [\tilde T] \ge \frac{\sharp T}{L n}$$ which concludes the corollary upon assuming that $G$ has purely exponential growth.
\end{proof}

\section{Growth tightness of   fractional barrier-free sets}\label{Section5}
In this and next sections, we assume that the action $G \curvearrowright \mathrm Y$ is SCC and prepare  two further ingredients {to prove} the upper bound on $\sharp \mathcal C(n)$. The first one, contained in this section, is required when the action is not cocompact. The second one in {the} next section will address the possibly unbounded kernel of the exact sequence (\ref{ExactEg}).

\subsection{Growth tightness for fractional  barrier-free sets}

We begin with a generalization of barrier-free elements.  See \textsection \ref{SSfracbarrier} for another formulation of this notion. 

Similar results in Teichm\"{u}ller spaces have been obtained   in the work of Dowdall, Duchin and Masur \cite{DDM}. However, none of {these} could be implied by the other, and our method uses  very little information from the theory of Teichm\"{u}ller spaces except {that the action of the mapping class group on Teichm\"{u}ller space is SCC with a contracting element.}
\begin{defn}
Fix $\theta\in (0,1]$, $\epsilon, M, L>0$ and $P\subset G$. Let $g\in G$ be an element. If there exists a set $\mathbb K$ of disjoint connected open subintervals  $\alpha$ of length at least $L$ in $[o, go]$ such that each $\alpha$ is $(\epsilon, P)$-barrier-free with two endpoints $\partial \alpha \subset N_M(Go)$ {and such that}
$$
 \sum_{\alpha\in \mathbb K} \len(\alpha)\ge \theta d(o, go),
$$
then  $g$ is said to satisfy a  \textit{$(\theta, L)$-fractional $(\epsilon, P)$-barrier-free} property.

Denote by  $\mathcal V^{\theta, L}_{\epsilon, M, P}$  the set of elements $g\in G$ with the   $(\theta, L)$-fractional $(\epsilon, P)$-barrier-free property.
\end{defn}
We are actually interested in the set of elements of $g\in G$ so that $[o, go]$ spends $\theta$-percentage of time in $N_M(Go)$. Here, we prove a more general statement for   { potential future applications}.

\begin{figure}[htb] 
\centering \scalebox{0.7}{
\includegraphics{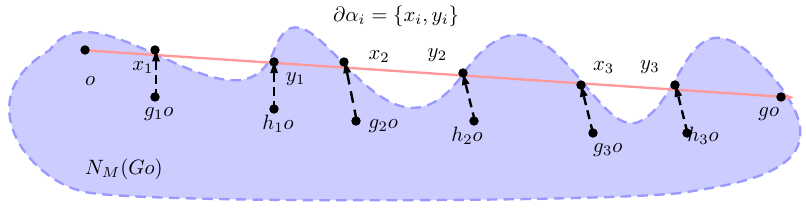} 
} \caption{An example: the set $\mathbb K$ consists of maximal connected components of length $\ge L$ outside $N_M(Go)$. See Corollary \ref{LargePercentGeneric}} \label{figure5}
\end{figure}

\begin{thm}\label{FGrowthTightThm}
Let $\epsilon, M>0$ be given by Theorem \ref{GrowthTightThm}.
For any $0<\theta\le 1$, there exists $L=L(\theta)>0$ such that $\mathcal V^{\theta, L}_{\epsilon, M, P}$  is a growth tight set for any    $P\subset G$.
\end{thm}
\begin{proof}
{Consider an element  $g\in \mathcal V^{\theta, L}_{\epsilon, M, P}$ so that $[o, go]$ cumulatively spends at least  $\theta$-percentage of time in $N_M(Go)$ with each stay having length at least $L$ and each stay being $(\epsilon, P)$-barrier-free}. Then the number $m=\sharp \mathbb K$ of stays is at most $\theta n/L.$

Let us denote  the endpoints of the $i$-th interval in $\mathbb K$     by $x_i, y_i$ for $1\le i\le m$ so that there exist elements $g_i, h_i\in G$ {with the property that} $$d(g_io, x_i), \;d(h_io, y_i)\le M.$$
By definition, we have $g_i^{-1}h_i\in \mathcal V_{\epsilon, M,  P}$. Let   $\omega_1>0$ be the growth rate of $\mathcal V_{\epsilon, M, P}$, which {by Theorem \ref{GrowthTightThm}} is strictly less than $\e G$. For some $\lambda_1>0$, we have $$\sharp \mathcal V_{\epsilon, M,  P}\cap N(o, n) \le \lambda_1 \exp(n\omega_1).$$
By Theorem B in \cite{YANG11}, there exists $\lambda_2>0$ such that the following holds for any proper action with a contracting element:
$$
\sharp N(o, n)\le \lambda_2 \exp(n\e G)$$
for $n\ge 1$.

Keeping the positions at $x_i, y_i$ as fixed along a geodesic of length $n$, the number of elements $g$ is {bounded above} by $$\lambda^{2m}\cdot \exp(n(1-\theta)\e G)\cdot \exp(n \theta  \omega_1)$$
where $\lambda>1$ is a constant  depending only on $\lambda_i$ and $M$.
Indeed, the second factor comes from the product of the number of elements in balls of radius $d(g_io, h_{i+1}o)\le d(y_i, x_{i+1})+2M$, and the third factor is the product of barrier-free elements $g_i^{-1}h_i$ corresponding to $x_i, y_i$.  The coefficient $\lambda^{2m}$ takes into account the constant $\lambda_i, M$ in $2m$ multiplications.  

Fixing $m\le n\theta/L$   and varying $x_i, y_i$,   the number of configurations is at most $C^{m}_n$, {bounded above} by 
$$\big(\frac{en}{m}\big)^{m } \le \big(\frac{eL}{\theta}\big)^{n\theta/L},$$
which follows from Stirling's formula.

Let $\omega_2 =\e G-\omega_1>0$. We deduce that  $\sharp \mathcal V^{\theta, L}_{\epsilon, M, P}$ is {bounded above} by   
\begin{align*}
 &n\theta/L\cdot \big (\frac{eL}{\theta}\big)^{n\theta/L} \cdot \lambda^{2m}\cdot \exp (-n\theta\omega_2 )\cdot \exp (n\e G)  \\
 \le  & n\theta/L\cdot \big(\frac{e L}{\theta} \lambda^{2}\big)^{n\theta/L}\cdot \exp (-n\theta\omega_2 )\cdot  \exp (n\e G)
\end{align*}
For given $\theta$, if $L$ is large enough, the factor $(\frac{e L}{\theta} \lambda^{2})^{n\theta/L}$ is little-o of $\exp(\sigma n)$ for some $\omega_2>\sigma>0$. Hence, the product of the first three factors is of order $\exp((\sigma-\omega_2) n)$. The growth tightness of $\mathcal V^{\theta, L}_{\epsilon, M, P}$ follows as desired.
\end{proof}

We now derive a corollary of Theorem \ref{FGrowthTightThm} in a specific setting.
Let $\theta\in (0, 1]$ and $L>0$. Given $g\in G$, consider    the set $\mathbb K$ of maximal connected components $\alpha$ of length at least $L$ in the intersection of $[o, go]$ {with the complement of $N_M(Go)$}. Let   $\mathcal O_{M}^{\theta, L}$ denote the set of elements $g$ with the following property:  
$$
 \sum_{\alpha\in \mathbb K} \len(\alpha)\ge \theta d(o, go).
$$

\begin{cor}\label{LargePercentGeneric}
For any $\theta\in (0, 1]$ there exists $L=L(\theta)>0$ such that 
$\mathcal O_{M}^{\theta, L}$ is a growth tight set.
\end{cor}
\begin{proof}
Fix a contracting element $f\in G$. To apply Theorem \ref{FGrowthTightThm}, it suffices to verify that each component $\alpha$ in $\mathbb K$ defined as above is $(\epsilon, f^n)$-barrier-free for some $n\gg 0$.   Indeed, if not, assume that $\alpha$ contains an $(\epsilon, f^n)$-barrier $t\in G$  for any  large $n$. By the contracting property, we see that if $n$ is chosen so that $d(o, f^no)$ is large enough, then $\alpha$ intersects a uniform $C$-neighborhood of the contracting set $t\ax(f)$, where $C$ is the contraction constant of $t\ax(f)$.  However, the component $\alpha$ is disjoint from $N_M(Go)$ by definition, so we have a contradiction when $M\ge C$. Hence, $\alpha$ is $(\epsilon, f^n)$-barrier-free for some $n\gg 0$. The proof is complete.
\end{proof}

\subsection{Upper bound for strongly primitive conjugacy classes}

As applications of previous results, we establish the desired upper bound on the number of strongly primitive conjugacy classes.
 
We fix a choice of constants $\theta\in (0, 1], L=L(\theta)>0$ so that Corollary \ref{LargePercentGeneric} holds. Thus, there exists an exponentially large set of elements $g\in G$ such that at least $(1-\theta)$-proportion of $[o, go]$ lies in the $(M+L/2)$-neighbourhood of $Go$. In other words, the cumulative stay  of $[o, go]$ in $N_{M+L/2}(Go)$ takes up at least $(1-\theta)$-percentage of the whole segment $[o, go]$.


\begin{lem} \label{ThickElements}
There exist an exponentially generic set $\mathcal G$ of elements   and $\theta, K>0$ with the following property. 
If $g \in \mathcal G$ is a strongly primitive contracting element then $[g]$ contains at least $\theta \cdot \tau[g]$ elements   in $N(o, \tau[g]+K)$.
\end{lem}
\begin{proof}
Let $\mathcal G$ be the set of elements $g\in G$ such that $g$ satisfies         Corollary \ref{GenericSetTwoLengths} and  does not belong to $\mathcal O_M^{\theta, L}(Go)$. Thus,  $\mathcal G$ is exponentially generic since it is the intersection of two such sets. 

Assume that  $g$  is strongly primitive  and  minimal in $[g]$. The next three paragraphs can be ignored, if the reader is interested only in the case of a co-compact action.

As discussed above, for $g\notin  \mathcal O_M^{\theta, L}(Go)$, at least $(1-\theta)n$-proportion of $[o, go]$ is contained in $N_{M+L/2}(Go)$. We plot inductively a sequence of  points  with step of length at least $2M$   so that the the total number $m$ of points is less than $(1-\theta)n/2M$.  

Precisely, from left to right, choose $x_0=o$  to start. If the point on $[o, go]$  with an exact distance $2M$ {from $x_0$ lies in $N_{M+L/2}(Go)$, let $x_1$ be this point; otherwise choose $x_1$ to be the closest to $x_0$ {among points in $N_{M+L/2}(Go)$ satisfying $d(x_1, x_0)\ge 2M$}}. We do this repeatedly so that $d(x_{i+1}, x_i)\ge K$ and $x_i\in N_{M+L/2}(Go)$, until the terminal $go$ is within $2M$-distance of $x_{m-1}$. Finally, set $x_m=go$. 

Observe  that $m\ge (1-\theta) n/{4M}$. {Indeed, the union of $2M$-neighborhoods of $x_i$'s covers $[o, go]\cap N_{M+L/2}(Go)$ so the lower bound on $m$ follows.}

We have thus subdivided $[o, go]$ into segments of length at least $2M$. Thus, for some $\theta'=\theta'(\theta)>0$, there exists $m=\theta' n$ elements $h_i$ such that $d(h_io,x_i)\le M+L/2$ and $d(h_io, h_{i+1}o)\ge 2M$. We can then write $g=s_1s_2\cdots s_m$ as a product of elements $s_i=h_i^{-1}h_{i+1}$ for $1\le i\le m$. 

The remainder of the proof is to show that all $m$ cyclic permutations of the word $s_1s_2\cdots s_m$ give  different elements $g_i=s_{i+1}\cdots s_n s_1\cdots s_i$ with length bounded by $d(o, g_io)\le d(o, go)+2M$. Since the strong primitivity is   a  conjugacy invariant and $g_i$ is conjugate to $g$, we have that each $g_i$ is   strongly primitive. 

First of all, we verify the upper bound: 
\begin{align*}
d(o, g_io)&\le d(o, s_1\cdots s_i o)+d(o, s_{i+1}\cdots s_n o)\\
&\le d(o, x_i)+M+d(x_i, go)+M\\
&\le d(o, go)+2M.
\end{align*}

Now it remains to prove that $g_i\ne g_j$ for $i \ne j$.   To derive a contradiction, assume that $i<j$ and $g_i=g_j=h$. Denoting $t=s_{i+1}\cdots s_j$, we have $ht= g_i t= t g_j=th$, so $t\in E^+(h)$ by Corollary \ref{EPlush}. The strong primitivity of $h$ then allows to write $t=h^k f$, where  $|k|>1$ and $f$ belongs to a finite normal subgroup $F$ of $E^+(h)$.  Thus, there exists some $n>0$ such that $t^n=h^{nk}$, so  $\tau[t]=|k|\tau [h]$ for $|k|\ge 2$.   

The stable length is a conjugacy invariant, so $\tau[h]=\tau[g]$. Since $g\in \mathcal G$ satisfies Proposition \ref{StablePointedLength}, there exists $K_0>0$ such that $$|\tau[g]-d(o, go)|\le K_0$$ yielding   $$\tau[t] \ge |k| (d(o, go)-K_0).$$

Recalling that $t=s_{i+1}\cdots s_j$, we obtain that 
\begin{align*}
\tau[t]&\le d(o, to) \le d(h_io, h_{j+1}o)\le d(x_i, x_{j+1})+2M \\
& \le d(o, go)-2M+2M\le d(o, go).
\end{align*}
This gives a contradiction when $d(o, go)$ is large enough. Accordingly, all cyclic permutations $g_i$ are distinct: there are at least $m=\theta' n$ elements in $N(o, d(o, go)+2M)$, where $\theta'$ is a uniform number. The lemma is thus proved.
\end{proof}

We also need the following lemma.
 
\begin{lem} \cite[Lemma 3.2]{CoK2}\label{SumExpOn}
 Let $\omega>0$. Then there exists $C>0$ such that 
 $$
\sum_{0\le i\le n} \frac{\exp(i \omega)}{i} \le C\frac{\exp(n \omega)}{n}. 
 $$
\end{lem} 

Denote by $\mathcal {C''}(n)$ the set of conjugacy classes of strongly primitive  contracting elements of stable length at most $n$.
 
\begin{cor}\label{PrimConjUpBnd}
Assume   the action $G \curvearrowright \mathrm Y$ is SCC. Then $$\sharp \mathcal {C''}(n)\cap \mathcal C(o, n)\prec \displaystyle\frac{\exp(\e Gn)}{n}.$$ 
\end{cor}
\begin{proof}
Define $$\mathcal A''(i, \Delta)=\{[g]\in \mathcal C(o, n)\cap  \mathcal {C''}(n): |\tau[g]-i|\le \Delta\}$$ for $0\le i\le n, \Delta\ge 0$. 
By Lemma \ref{ThickElements}, each  class $[g]\in \mathcal A''(i, \Delta)$ of strongly primitive contracting elements  contains at least $\theta (i+\Delta)$ elements in $N(o, i+K)$ for some uniform $K>0$. Since $\sharp N(o, i)\asymp \exp(\e Gi)$, we obtain $\sharp \mathcal A''(i, \Delta) \prec  \frac{\exp(\e Gi)}{i}$.   The proof is completed by Lemma \ref{SumExpOn}.
\end{proof}

\section{Uniform contracting elements}\label{Section6}

We now discuss the second ingredient to deal with torsion elements   in the kernel of the following exact sequence
$$
1\to F\to E^+(g)\to \mathbb Z\to 1
$$ 
The main technical result of independent interest is that an {exponentially} large set of contracting elements have a uniform bound on the kernel $F$.  This shall be proved using Bestvina-Bromberg-Fujiwara's construction of projection complex.

We point out that the results in this section {is} not necessary, if the action has no torsion elements or there exists a uniform bound on the size of finite subgroups. The latter {condition} is satisfied in large varieties of groups: hyperbolic groups, CAT(0) groups and mapping class groups, to point out a few. However, there exist  examples of relatively hyperbolic groups which have {finite subgroups of unbounded cardinality}.

\subsection{Projection complex}
Let $\mathbb X$ be a collection of uniformly contracting sets with bounded intersection. By \cite{Sisto}, $\mathbb X$ satisfies the projection complex axioms introduced in \cite[Section 3.1]{BBF}. We recall some of their results that will be useful for us.

For a constant $K>0$, the \textit{projection complex}  $\mathcal P_K(\mathbb X)$ is a graph obtained in the following way. They introduced a function $\tilde d_W(X_1, X_2)$ for every     $X_2\ne W\ne X_1\in \mathbb X$, which is variant of  $  d_W(X_1, X_2)=\diam{\pi_W(X_1\cup X_2)}$ so that 
\begin{equation}\label{BBFDistEQ}
d_W(X_1, X_2)\simeq_{\theta} \tilde d_W(X_1, X_2)
\end{equation} where $\theta>0$ depends only on $\mathbb X$. The vertex set consists of all elements in $\mathbb X$, and two distinct vertices $X_1, X_2$ are connected if and only if $\tilde d_W(X_1, X_2)< K$ for every $W\in \mathbb X$. Introduce the interval-like set 
$$\mathbb X_K(V, W)=\{X\in\mathbb X: \tilde d_X(V, W)\ge K\}.$$ Thus, two vertices $X_1, X_2$ are adjacent if and only if $\mathbb X_K(X_1, X_2)=\emptyset$. The basic result in \cite{BBF} is that for a large $K>0$, the  projection complex $\mathcal P_K(\mathbb X)$ is a quasi-tree of infinite diameter on which $G$ acts co-boundly. 
 
 It is proved in \cite[Prop 3.7]{BBF}  that the interval $\mathbb X_K(V, W)$ gives a path in the projection complex between $V$ and $W$ which is not necessarily geodesic. However, if raising $K$ to a large amount, we have the following result. 
\begin{lem}\cite[Lemma 3.18]{BBF} \label{StrongGeod}
For a sufficiently large  $K'$ relative to $K$, all the elements in $\mathbb X_{K'}(V, W)$ indeed appear in any geodesic between $V, W$.
\end{lem} 

The class of acylindrical actions has received great interest in recent years. See Osin \cite{Osin6} and references therein, for a survey of a rapidly growing body of studies of acylindrical hyperbolic groups. 

\begin{defn}\label{HypAclindDef}
An action of $G$ on a geodesic metric space $\mathrm Y$ is \textit{acylindrical} if for any $D>0$ there exist $R, N>0$ such that  
$$
\sharp \{g\in G: d(x, gx)\le D, d(y,  gy)\le D, d(x, y)>R\} \le N
$$
for any given $x, y\in \mathrm Y$. 
\end{defn}

{The action of $G$ on the projection complex in \cite{BBF} was recently shown to be acylindrical for a variety of interesting examples, including a proper action with contracting elements.}
\begin{thm}\cite[Theorem 5.10]{BBFS}\label{AcylinAction}
There exists $K>0$ such that $G$ acts acylindrically on the projection complex  $\mathcal P_K(\mathbb X)$ that is a quasi-tree.
\end{thm}
\begin{proof}
This is stated in \cite[Theorem 5.10]{BBFS}, as a consequence of \cite[Theorem 3.9]{BBFS} applied in our setting. However, no explicit proof is given there to verify the following condition: for some fixed $N$ and $B$, for any $N$ distinct elements of any $\mathbb X_K(V, W)$ the common
stabilizer is a finite subgroup of size at most $B$. 

Indeed, any fixed integer $N\ge 2$ suffices. If $g$ lies in the stabilizer of a  set of $N$ elements, say containing distinct $V=t_1\ax(f), W=t_2\ax(f) \in \mathbb X$,  then a   power $g^k$  fixes $V$ and $W$ for some $k=k(N)$. By definition of $\ax(f)=E(f)\cdot o$ and $E(f)$ (\ref{Ehdefn}), we have that $E(f)$ is the stabilizer of $\ax(f)$ and so $g^k\in t_1 E(f)t_1^{-1} \cap t_2 E(f)t_2^{-1}$. It suffices to see that $t E(f)t^{-1} \cap E(f)$ contains at most $B$ elements for any $t\in G\setminus E(f)$. {Since the action is proper}, there are only finitely many $tE(f)$ such that the intersection  is non-trivial.  For each $tE(f)\ne E(f)$, the intersection $t E(f)t^{-1} \cap E(f)$ is finite. So we obtained a uniform number $B_0$ bounding on $\sharp t E(f)t^{-1} \cap E(f)$ for any $t\notin E(f)$. Thus, $B=kB_0$ is the desired constant in the above condition.
\end{proof}

Denote by $d_{\mathcal P}$ the induced length metric on the graph $\mathcal P_K(\mathbb X)$. The following criterion {for an isometry to be  loxodromic} was obtained in \cite[Lemma 3.22]{BBF}.
\begin{lem}\label{LoxoElem}
Let $K'$ be the constant given in Lemma \ref{StrongGeod}. If there exists $N>0$ and $X\in \mathbb X$ such that $d_{\mathcal P}(g^{-N} X, g^N X)>K'$, then $g$ acts as a loxodromic isometry on $\mathcal P_K(\mathbb X)$. 
\end{lem}

\subsection{Generic elements act loxodromically on projection complex}

In this subsection, we fix two independent  contracting elements $f_1, f_2$. Then  the set $\mathbb X=\{g\ax(f_i): g\in G, i=1, 2\}$ is a contracting system with bounded projection, for which  we   construct  the projection complex $\mathcal P_K(\mathbb X)$ with $K$ satisfying Theorem \ref{AcylinAction}.

We are now ready to state the main result of this section.

\begin{lem}\label{LoxdroPC}
There exists an exponentially generic set $\mathcal G$ of elements which act by loxodromic isometries on the projection complex $\mathcal P_K(\mathbb X)$.
\end{lem}
\begin{proof}
Fix a large integer $m>0$. Let $\mathcal G$ be the set of elements $g$ such that the minimal representative in $[g]$ contains at least one $(\epsilon, f_1^m)$-barrier and one $(\epsilon, f_2^m)$-barrier. By Theorem \ref{Conj2BarrierFree}, this is an exponentially generic set.

We shall apply Lemma \ref{LoxoElem} to show that    $g\in \mathcal G$ is a loxodromic isometry on $\mathcal P_K(\mathbb X)$.  Since  loxodromic   elements are preserved under conjugacy, we may assume that $g$ is a minimal element in $[g]$. Denote $\alpha=[o, go]$. In the proof of Proposition \ref{StablePointedLength}, it is proved that there exists a periodic $(g, D, \sigma)$-admissible path $$\gamma=\cup_{i\in \mathbb Z} g^i([x, y]_\alpha \cdot [y, gx])$$ defined in (\ref{PeriodicAdmEQ}), where $x, y$ are the entry and exit points of $\alpha$ in $N_\epsilon(\mathbf b)$, and $\mathbf b$ is either an  $(\epsilon, f_1^m)$-barrier or an $(\epsilon, f_2^m)$-barrier. The associated combinatorial axis is $\mathbb A(\gamma)=\{g^i N_\epsilon(\mathbf b): i\in \mathbb Z\}$ and let us denote $X=N_\epsilon(\mathbf b)\in \mathcal P_K(\mathbb X)$.    

Consider the geodesic $[g^{-N} X, g^{N} X]$ in $\mathcal P_K(\mathbb X)$. We choose $m$ big enough so that $$D>\max \{d(o, f_io): i=1, 2\}-2\epsilon\gg K'$$  where the constant $K'$ is given by Lemma \ref{StrongGeod}. By Proposition \ref{admissible},  we can derive that for each $g^iX$ with $|i|< N$, we have $d_{g^iX}(g^{-N} X, g^{N} X)>K'$. By Lemma \ref{StrongGeod}, we have that $g^iX\in [g^{-N} X, g^{N} X]$. This implies that the length of the geodesic $[g^{-N} X, g^{N} X]$ grows linearly as $N\to \infty$.  Thus, the assumption of Lemma \ref{LoxoElem}  is fulfilled, so $g$ has positive stable length in $\mathcal P_K(\mathbb X)$. 
\end{proof}

The following consequence will be useful in the next section.

\begin{lem}\label{UnifKernel}
 There exists an integer $N>0$ such that for any $g\in \mathcal G$ (in Lemma \ref{LoxdroPC}), the following exact sequence 
$$
1\to F\to E^+(g)\to \mathbb Z\to 1
$$
where $\sharp F<N$.
\end{lem}
\begin{rem}
The lemma is sharp in the sense that there is no uniform bound on the kernel for a certain sequence of hyperbolic (contracting) elements in Dunwoody's inaccessible groups \cite{Abbott}. Note that, Dunwoody's groups are infinitely-ended and thus relatively hyperbolic, so the prime conjugacy growth formula holds by Corollary \ref{RHThm}.
\end{rem} 

\begin{proof}
The proof is due to \cite[Lemma 6.8]{Osin6} and is short, so we include it for completeness. Pick any point $x\in \mathrm Y$. Since $\mathrm Y$ is a hyperbolic space, the diameter of the coarse center of $F\cdot x$ is uniformly bounded by a constant $D>0$. Without loss of generality, assume that $x$ lies in the center so $d(x, fx)\le D$ for any $f\in F$.

Choose $h\in E^+$ such that $d(x, hx)\le R$. Since $h^{-1}fh\in F$ and $d(fh x, hx)=d(h^{-1}fhx, x) \le D$, the definition of acylindrical hyperbolicity implies  that the cardinality of $F$ is   bounded above by $N$.
\end{proof}

\section{Primitive       conjugacy classes are generic: end of the proof of \ref{mainthm}}\label{Section7}

To derive the upper bound for all conjugacy classes, we will show that  the non-primitive ones are exponentially negligible.

\begin{lem}\label{GenericPrim}
Assume that the action is SCC. Then 
$$\frac{\sharp \big(\mathcal {C}(o, n)\setminus \mathcal {C'}(o, n)\big)}{\sharp\mathcal {C}(o, n)}\to 0,$$
exponentially quick as $n\to\infty$.  
\end{lem}
\begin{proof}
By Lemma \ref{TightConjugacyBarrier}, we can assume that all elements in $\mathcal C(o, n)$ are  contracting. By Lemma \ref{UnifKernel},   assume further that for each $[g]\in \mathcal C(o, n)$, there  exists a uniform number $N>0$ with the following exact sequence 
$$
1\to F\to E^+(g) \stackrel{\phi}{\to}  \mathbb Z\to 1
$$
where $\sharp F<N$. By Corollary \ref{GenericSetTwoLengths}, we can assume that $$|\tau[g]-\len_o[g]|\le B$$  for a uniform constant $B>0$. 

By Lemma \ref{StrongPrimitive}, a non-primitive contracting element is not strongly primitive. Thus, it suffices to bound  elements which are not strongly primitive. Let $\mathcal C''(o, n)$ be  the   set of conjugacy classes of strongly primitive elements with  algebraic  length at most $n$. 
 
Let $[g]\in \mathcal C(o, n)\setminus \mathcal C''(o, n)$ be a non-strongly primitive element, so by definition, there is a strongly primitive element $g_0\in E^+(g)$ so that $g=g_0^m$ for some $|m|\ge 2$. Note that there are at most $\sharp F$ choices of $g_0$, where $\sharp F<N$.
 
By    Corollary \ref{GenericSetTwoLengths}, $$\len_o[g_0]\le B+\frac{\tau[g]}{|m|}\le B+n/|m|.$$ We define  a map as follows $$\Pi: \mathcal C(o, n)\setminus \mathcal C''(o, n) \to \mathcal C''(o, n+B)$$ 
by $\Pi([g]):=[g_0] \in \mathcal C''(o, \frac{n}{|m|}+B).$

For $|m|\ge 2$, we obtain a constant $0<\omega_1<\e G$ such that  
$$\sum_{|m|\ge 2} \sharp \mathcal C''(o, \frac{n}{|m|}+B) \le \sum_{|m|\ge 2} \sharp N(o, \frac{n}{|m|}+B) \le \exp(\omega_1 n).$$ 

Recall that for each   $n\ge |m|\ge 2$, there are at most $N$ conjugacy classes $[g]$ such that $\Pi([g])= [g_0]$ with $\phi(g)=m$.  Hence,  $$\sharp \mathcal {C}(o, n)\setminus \mathcal {C''}(o, n) \le  N n \exp(\omega_1 n).$$   Since $\mathcal {C''}(o, n)\subset \mathcal {C'}(o, n)$ and $\sharp \mathcal C(o, n)\succ \frac{\exp( \e Gn)}{n}$ by  Corollary \ref{LbdConjGrowth}, the conclusion follows.
\end{proof}
 
Combined with Corollary \ref{PrimConjUpBnd}, we obtain the upper bound for all conjugacy classes.  
\begin{cor}\label{ConjUpBnd}
Assume that the action is SCC. Then $$\sharp \big(\mathcal {C}( n)\cap \mathcal {C}(o, n)\big)\le \sharp \mathcal {C}(o, n)\prec \displaystyle\frac{\exp(\e Gn)}{n}.$$
\end{cor}

\begin{lem}\label{GenericPrim2}
Assume that the action is SCC. Then 
$$
\frac{\sharp \big( \mathcal {C}(o, n)\setminus \mathcal {C'}( n)\big)}{\sharp \big(\mathcal {C}(n)\cap \mathcal {C}(o, n)\big)}\to 0$$
exponentially quick as $n\to\infty$.  
\end{lem}
\begin{proof}
Denote by $\mathcal C''(n)$ the set of conjugacy classes of strongly primitive elements with stable length at most $n$. Noting that $\mathcal C''(o, n)\subset \mathcal C''(n)$,  the proof of Lemma \ref{GenericPrim} shows  $$\frac{\sharp \big (\mathcal {C}(o, n)\setminus \mathcal C''(n)\big)}{\sharp \mathcal C(o, n)}\to 0$$   exponentially quick.  
By Corollaries \ref{LbdConjGrowth} and \ref{ConjUpBnd}, we have $\sharp \big(\mathcal C(n)\cap \mathcal C(o, n)\big) \succ \sharp \mathcal C(o, n)$. 
We then obtain that $\mathcal {C}(o, n)\setminus \mathcal C''(n)$ is exponentially small relative to $\mathcal C(n)\cap \mathcal C(o, n)$. Thus, the conclusion follows.
\end{proof}

By Corollary \ref{LbdConjGrowth}, we obtain the lower bound for primitive conjugacy classes.
\begin{cor}\label{LbdPrimConjGrowth}
Assume that the action is SCC. Then $$\sharp \mathcal {C'}(o, n)\ge \sharp \big(\mathcal {C'}( n)\cap \mathcal {C}(o, n)\big)\succ \displaystyle\frac{\exp(\e Gn)}{n}.$$ 
\end{cor}

Therefore, all assertions in \ref{mainthm} are proved.

\section{Applications}\label{SecProofs}

\subsection{Application to Teichm\"{u}ller space}

See e.g. \cite{Masur}, \cite{Wright} for background on Teichm\"{u}ller theory. 

Let $S$ be a closed oriented surface of genus at least 2.
The Teichm\"{u}ller space $Teich(S)$ is the space of marked conformal structures on $S$. The Teichm\"{u}ller metric is given by $$d(x,y)=\frac{1}{2} \inf_{f} \log K(f)$$ where $K(f)$ denotes the quasiconformal constant and $f$ varies over a given homotopy class.
This makes $Teich(S)$ into a proper unique geodesic metric space, and in fact a Finsler manifold. The unit (co)tangent bundle of $Teich(S)$ may be identified with
the space $Q(S)$ of unit area holomorphic quadratic differentials. By slight abuse of notation, we consider a Teichm\"{u}ller geodesic as a subset of both $Q^{1}(S)$ and $Teich(S)$.

The principal stratum $Q_{top}$ consists of those quadratic differentials all of whose zeros are simple. The mapping class group $MCG(S)$ acts on $Teich(S)$ by isometries. An element  $f\in MCG(S)$ is contracting for this action if and only if it is pseudo-Anosov. In this case, it preserves an invariant Teichm\"{u}ller geodesic: the axis $\Teichax(f)$ of $f$. The stable    length $\tau(f)$ of $f$ is the displacement of $f$ along this axis.

Gadre and Maher prove the following \cite[Proposition 2.7]{Gadre-Maher}.

\begin{prop}\label{closeprincipal}
Let $f\in MCG(S)$ be a pseudo-Anosov mapping class such that $\Teichax(f)$ lies in the principal stratum. For each $K>0$ there is an $L=L(S,f)>0$ such that whenever $\gamma$ is a bi-infinite Teichm\"{u}ller geodesic with uniquely ergodic vertical and horizontal measured foliations and containing points $X_{1},X_{2}$ in Teichm\"{u}ller space with $d(X_1,X_2)>L$ and $d(X_i,ax(f))<K$ for $i=1,2$ it follows that $\gamma$ also lies in the principal stratum. 
\end{prop}

We now prove Theorem \ref{Teichprincipal} from the introduction.
\begin{thm}(Theorem \ref{Teichprincipal})
Let $o\in Teich(S)$ and $\theta_1, \theta_2 \in (0,1)$.
The subset of $g\in MCG(S)$ of pseudo-Anosov elements satisfying the following properties is exponentially generic with respect to the Teichm\"{u}ller metric: 
\begin{enumerate}
\item

$d(o,go)\geq \tau(g)>\theta_1 d(o,go)$
\item

 $d(o,\Teichax(g))<(1-\theta_2)d(o,go)$
\item
 $\Teichax(g)$ lies in the principal stratum
\end{enumerate}
\end{thm}

\begin{proof}
This essentially follows from Theorem \ref{simplifiedLinearDrift} and Proposition \ref{closeprincipal}.   
The only point that requires clarification is that the set of elements satisfying (3) is exponentially generic. Indeed, assume without loss of generality that $o\in \Teichax(f)$.
Theorem \ref{simplifiedLinearDrift} (3) implies that for each $n>0$ the set $A(f,n,K)$ of pseudo-Anosov $g$ such that $\Teichax(g)$ contains an $(K, f^n)$-barrier is exponentially generic. This means for each $g\in A(f,n,K)$  there is an $h\in MCG$ with
$d(o,h^{-1}\Teichax(g))<K$ and $d(f^n o,h^{-1}\Teichax(g))<K$. Note, $h^{-1}\Teichax(g)=\Teichax(h^{-1}gh)$. Thus $\Teichax(h^{-1}gh)$ contains $X_1, X_2$ with $d(X_1,o)<K$, $d(X_1,f^n o)<K$. Hence, $d(X_i, \Teichax(f))<K$ and $d(X_1,X_2)>d(o,f^n o)-2K=n\tau(f)-2K$. For large enough $n$   we have $n\tau(f)-2K$ is greater than the constant $L=L(S,f)$ from Proposition \ref{Teichprincipal} so that $\Teichax(h^{-1}gh)$ lies in the principal stratum. Since the principal stratum is $MCG(S)$ invariant it follows that $\Teichax(g)$ lies in the principal stratum.
\end{proof}

\subsection{New examples with prime conjugacy growth formulae}
In this subsection, we give some detail about examples  appearing in Corollaries \ref{CAT0Thm}, \ref{RHThm}, \ref{ModThm}. 

We first  consider the class of cubical groups. A group $G$ is \textit{cubical} if it admits a geometric action on a CAT(0) cube complex $\mathrm Y$. One the one hand, when $\mathrm Y$ is endowed with CAT(0) metric, an isometry is called \textit{rank-1} if $\tau(g)>0$ and any axis of it does not bound a half-flat. By \cite[Theorem 5.4]{BF2}, a rank-1 isometry is exactly a contracting element with respect to CAT(0) metric. 

On the other hand, it is very useful to study the cubical geometry when the one-skeleton of $\mathrm Y$ is equipped with the combinatorial metric. This is an $\ell^1$-metric, in contrast with $\ell^2$-metric induced from CAT(0) metric.    The following result is certainly known to experts \cite{Huang}.
\begin{lem}\label{rank1l1}
Let $G \curvearrowright \mathrm Y$ be a cubical group such that $\mathrm Y$ does not factor as a product of unbounded cube subcomplexes. Then $G$ contains a {contracting} element with respect to the action on one-skeleton of $\mathrm Y$ with the $\ell^1$-metric. Moreover, every rank-1 element preserves an $\ell^1$-geodesic by translation.
\end{lem}
\begin{proof}
Two disjoint hyperplanes $H_1, H_2$ are called \textit{$k$-separated} for $k\ge 0$ if there are at most $k$ hyperplanes intersecting both $H_1$ and $H_2$. An element $g$ \textit{skewers} $H_1, H_2$ if it pushes one halfspace bounding by $H_1$ into one bounding by $H_1$. By \cite[Theorem 6.3]{CapSag}, there exists a contracting isometry $g$ in $\ell^2$-metric  so that it skewers   $0$-separated hyperplanes. Such a hyperplane skewering involves no metrics at all and thus implies that $g$ is contracting with respect to $\ell^1$-metric by  \cite[Theorem 3.13]{Genevois}. The conclusion follows. 

In fact, a contracting element in $\ell^2$-metric is exactly contracting in $\ell^1$-metric. 
It is easy to see that an element $g$ is a rank-1 element in $\ell^2$-metric  iff it skewers a pair of $k$-separated hyperplanes for some $k\ge 0$. The  direction ``$\Leftarrow$'' is given by  \cite[Lemma 6.2]{CapSag}, and the other direction follows from  \cite[Lemma 4.6]{ChaSul}. Then the proof is concluded by \cite[Theorem 3.13]{Genevois} that the same hyperplane relation characterizes the  contracting property in $\ell^1$-metric.  The ``moreover'' statement is proved by Haglund in \cite{Hag}, since a contracting element acts without inversions on cube complex $\mathrm Y$ by Lemma \ref{GeodesicNearAxis}. 
\end{proof}

\paragraph{\textbf{Right-angled Artin (Coxeter) groups}} 
The class of right-angled Artin groups (RAAGs) is presented by
\begin{equation}\label{RAAG}
G=\langle V(\Gamma)|v_1v_2=v_2v_1 \Leftrightarrow (v_1, v_2)\in E(\Gamma)  \rangle
\end{equation}
for a finite simplicial graph $\Gamma$. See \cite{Charney} and \cite{Kob} for  references on RAAGs. A RAAG acts properly and cocompactly on a CAT(0)  cube complex called the \textit{Salvetti complex}.  In \cite[Theorem 5.2]{BehC}, it is proved that if $G$ is not a direct product, then it  contains a  rank-1  element.

The class of right-angled Coxeter groups (RACGs) is defined as in (\ref{RAAG})   with additional relations $v^2=1$ for each $v\in V(\Gamma)$. A RACG also acts properly and cocompactly on a CAT(0) cube complex called the Davis complex. From \cite[Proposition 2.11]{BHS} and   \cite[Theorem 2.14]{ChaSul}, a RACG $G$ is not virtually a direct product of groups if and only if it contains a rank-1 element. 

Using the co-compact action, it is clear that if an isometry preserves a geodesic by translation, then the stable length coincides with algebraic length up to a uniform error.
Together \ref{mainthm},    the assertions (1) (2) in Corollary \ref{CAT0Thm} follow  by Lemma \ref{rank1l1}.
  
\paragraph{\textbf{Relatively hyperbolic groups}}

In a relatively hyperbolic group, hyperbolic elements are contracting with respect to the action on the Cayley graph, cf. \cite{GePo4}, \cite{GePo2}. To count conjugacy classes of hyperbolic elements in stable length, we need the following fact.

\begin{lem}
There exists a uniform constant $B>0$ such that for every hyperbolic element $h$, we have $|\tau[h]-\ell_o[h]|\le B.$
\end{lem}
\begin{proof}[Sketch of the proof]
We refer to \cite{YANG7} for related notions and \cite[Proposition 7.8]{DuGe} for a similar argument.
Let $h_-, h_+$ be the fixed points of $h$ in the Bowditch boundary. By \cite[Lemma 2.20]{YANG7}, there exists a sequence of $(\epsilon, R)$-transitional points along any bi-infinite geodesic $\gamma$ between $h_-, h_+$, where $\epsilon, R$ are uniform depending only on $(G, \mathcal P)$. Moreover, if pickup any such transitional point $x$, then $\langle h\rangle x$ is uniformly close to $\gamma$. This implies that $|\tau[h]-\ell_o[h]|$ is uniformly bounded by an argument in Lemma \ref{StableAxisLength}. 
\end{proof}

By \ref{mainthm}, the prime conjugacy formula for hyperbolic elements with stable length follows and Corollary \ref{RHThm} is proved. 

\paragraph{\textbf{SCC covers of moduli spaces}}
We will now explain the conjugacy formula for certain cover of moduli spaces associated to subgroups constructed in \cite[Proposition 6.6]{YANG10}. For the convenience of the reader, we briefly explain the construction.

The subgroup  $H<MCG(S)$ is generated by a free abelian group $A$ generated by Dehn twists and a cyclic group generated by a pseudo-Anosov element $p$. In \cite[Proposition 6.6]{YANG10}, it is proved that $H$ is   a free product $A \star \langle p\rangle$. In order to prove this,    an admissible path in Teichm\"{u}ller space is constructed for each word in the free product $(A\star \langle p\rangle)\setminus (A\cup  \langle p\rangle)$. The admissible path travels alternatively along a sufficiently long segment of the (translates of) axis of $p$. This description shows that the word represents a pseudo-Anosov element in $H$, with the axis uniformly close to a translation axis of $p$.  In other words, each closed geodesic on the cover $Teich(S)/H$ of moduli space intersects  a finite neighborhood of the closed geodesic associated to $p$.

Since each pseudo-Anosov element in $H$ admits an axis uniformly close to the basepoint $o$, it follows that the stable length is within uniform additive error from the algebraic length. Thus, Corollary \ref{ModThm} now follows from the second item of \ref{mainthm}. The same proof also gives the general theorem \ref{CompactSupport}.

\subsection{SCC actions with infinitely many closed geodesics of bounded length}\label{SCCinfty}
We now give a geometrically infinite hyperbolic 3-manifold with infinitely many closed geodesics of bounded length and with a SCC deck transformation group action. This is essentially due to Peign\'e in \cite{Peigne3}, where geometrically infinite examples with finite Bowen-Margulis-Sullivan measure are constructed as the Schottky product of two Kleinian groups. For the convenience of the reader, we give some details about the construction. 

Let $H$ be an infinitely generated discrete group acting on the hyperbolic plane $\mathbb H^2$ such that $\e H\le 1$ and $\mathbb H^2/H$ contains infinitely many closed geodesics with bounded length.  Such subgroups can be chosen as the fundamental group of an infinite regular cover of a compact hyperbolic surface.

Via the Poincar\'e extension, the group $H$ acts on $\mathbb H^3$ with the limit set $\Lambda H$ contained in a circle. We now take a discrete group $G$ of divergent type  acting on $\mathbb H^3$ with $\e G>1$ and the limit set $\Lambda G$ disjoint from $\Lambda H$. Note that $H$ (resp. $G$) acts properly discontinuously outside $\Lambda H$ (resp. $\Lambda G$). By taking sufficiently deep finite index subgroups, we can have the following property: any $h\in H$ maps $\Lambda G$ into a small neighborhood of $\Lambda H$ and the same for $g\in G$. Thus, $H$ and $G$ stay at a ping-pong position so they generate a free product $\Gamma=G\star H$. 

The (infinitely generated) $\Gamma$ acts on $\mathbb H^3$ with infinitely many closed geodesics of bounded length on $\mathbb H^3/G$. Since $G$ is of divergent type,  we have $\e \Gamma > \e G$. Thus, by \cite[Proposition 6.3]{YANG10}, $\Gamma$ acts by a SCC action on $\mathbb H^3$.
 
The above construction applies in higher dimension, but we do not know of a finitely generated example with these properties. By Ahlfors measure theorem,  finitely generated examples cannot exist in dimension 3. (cf. \cite[Introduction]{Peigne}). 

Moreover, such examples could be produced in the  acylindrical actions on (possibly non-proper) hyperbolic spaces. 

At last, we would like to show that such phenomenon exists in the class of groups acting acylindrically on hyperbolic spaces (termed \textit{acylindrically hyperbolic groups} by Osin \cite{Osin6}). This is  pointed out  by the referee that  Watanabe \cite[Theorem 1.3]{Wat}  constructed infinitely many pseudo-Anosov mapping classes with same stable length on curve graphs (since the stable length spectrum is discrete by Bowditch \cite{Bow6}).  
\begin{lem}\label{InftyConjClassesLem}
Let $G$ be a non-elementary group acting acylindrically on a  hyperbolic space $X$ so that some infinite subgroup $H$ has bounded orbits. Then there exists infinitely conjugacy classes of loxodromic elements in $G$ with   bounded stable length.  
\end{lem}
\begin{proof} 
Fix a basepoint $x\in X$ so that $Hx$ has diameter at most $D>0$. Let $R>0$ be the constant in Definition \ref{HypAclindDef}. Let $f$ be a loxodromic element and $E(f)$ be the maximal elementary subgroup of $G$ containing $f$. Then the quasi-geodesic $\gamma:=E(f)x$ gives the quasi-axis of $f$. The following is the key claim.  There are infinitely many distinct $E(f)$-cosets $g_n\cdot E(f)$ over $g_n\in H$ so that $g_n\gamma, g_m \gamma$ fellow travel within $D$-distance at most $R$ for any $g_n\ne g_m$. Indeed, if not, we choose a point $y\in g_1\gamma$ such that $d(x, y)>R$.  The $R$-long fellow travel of $g_n\gamma$ and $g_1\gamma$ implies that $d(g_nx,x), d(g_ny,y)\le D$ holds for infinitely many $g_n$. This contradicts to the definition of acylindrical action.  

Since $h_n\gamma$ and $\gamma$ diverge after the time $R$, it is easy exercise in hyperbolic geometry that there exists a big power $\tilde f$ of $f$ (relative to $R$) such that $h_n:=\tilde f g_n $ are loxodromic elements for $n\ge 1$. Similar to periodic admissible paths, the quasi-axis denoted by $\gamma_n$ of $h_n$ has the fundamental domain $[o, \tilde f o][\tilde fo,\tilde f g_no]$, where $[o, g_no]$ has bounded length by $D$. Hence, all $h_n$ have uniformly bounded stable length. 
 
It remains to see that $\{h_n\}$ are contained in infinitely many distinct conjugacy classes. In fact, we shall show that no two $h_n\ne h_m$ lie in the same conjugacy class.  Indeed, if $h'h_nh'^{-1}=h_m$ for some $h'\in G$, consider the configuration  of the axis  $\gamma_m$ and $h'\gamma_n$ as in Figure \ref{figure3} (with $g$  replaced with $g_n$, $f$ with $\tilde f$, etc). Since $d(x, g_nx)\le D$, we conclude that $h'$ must send subsegments  $[x, \tilde fx]$ and $\tilde fg_n [x, \tilde f x]$ of $\gamma_n$ into a uniform neighborhood of $[x, \tilde fx]$ and $\tilde fg_m[x, \tilde fx]$ in $\gamma_m$. It is known that if the action is acylindrical, the set of translated quasi-axis $\{g\gamma: g\in G\}$ has bounded intersection. This implies $h'\gamma=\gamma$ and $h'\tilde fg_n\gamma=\tilde f g_m\gamma$. Since $\tilde f\in E(f)$ and $\gamma=E(f)x$, we see  $h'\in E(f)$ and thus $g_nE(f)=g_m E(f)$.  This contradicts to the choice of $g_n$ in different $E(f)$-cosets. The proof is complete.
\end{proof}

\subsection{Transcendental growth series} Theorem \ref{CongugacySeries} is a consequence of the prime conjugacy growth formula. 
 
 Define $\mathcal A(n, \Delta)=\{[g]\in \mathcal C: |\len_o[g]-n|\le \Delta\}$ for $n, \Delta\ge 0$. Each conjugacy class $[g]$ is contained in a uniform number of annuli {sets} $\mathcal A(n, \Delta)$ where $n>0$. Hence,  for fixed $\Delta>1$,  
$$
\mathcal P(z)=\sum_{[g]\in  G}     z^{\len_o[g]}  \asymp_\Delta \sum_{n\ge 0}   \sharp  \mathcal A(n, \Delta) z^n.
$$ 
By \ref{mainthm}, the prime conjugacy growth formula holds for the annulus 
$$
\sharp \mathcal A(n, \Delta) \asymp_\Delta \frac{\exp(n \e G)}{n}.
$$

The proof of  Theorem \ref{CongugacySeries} follows immediately from {the} asymptotics of the coefficients of an algebraic growth series in \cite[Theorem D]{Flaj}. See the proof of \cite[Theorem 1.1]{AntCio} for relevant details.



\bibliographystyle{amsplain}
 \bibliography{bibliography}

\end{document}